%% file: main.tex
\documentclass[letterpaper,11pt]{amsart}
\usepackage[margin=1.1in]{geometry}
\usepackage{amsmath,amsthm,amssymb}
\usepackage{xspace,xcolor}
\usepackage[breaklinks,colorlinks,citecolor=teal,linkcolor=teal,urlcolor=teal,pagebackref,hyperindex]{hyperref}
\usepackage[alphabetic]{amsrefs}
\usepackage[all]{xy}
\usepackage[english]{babel}
\usepackage{enumitem}
\usepackage{tikz,xcolor}
\usepackage{tikz-cd}
\usepackage{mathrsfs}
\usepackage{color}

\newtheorem{thmx}{Theorem}

\numberwithin{equation}{section}

\newcommand{\dX}{d_X}
\newcommand{\dZ}{d_Z}
\newcommand{\dD}{d_D}

\input{Preamble}
\begin{document}

\title[Microlocal Bernstein--Sato polynomials on singular ambient varieties]{Microlocal Bernstein--Sato polynomials on singular ambient varieties}
\author{Bradley Dirks}
\address{Department of Mathematics, Stony Brook University, Stony Brook, NY 11794-3651, USA}
\email{bradley.dirks@stonybrook.edu}
\author{Michael Perlman}
\address{Department of Mathematics, The University of Alabama, Tuscaloosa, AL 35401 USA}
\email{mperlman@ua.edu}


\thanks{B.D. was partially supported by the National Science Foundation under MSPRF DMS-2303070.}

\date{}


\begin{abstract}
We introduce the microlocal Bernstein--Sato polynomial of a function on a possibly singular ambient variety, extending the theory of Saito. We show that, contrary to the smooth ambient setting, these polynomials are not generally equal to the reduced $b$-functions obtained by removing the trivial root. We define the minimal exponent and use it to study the singularities of the divisor and the Hodge filtration on local cohomology. Our main results include a generalization of Saito's theorem relating the minimal exponent to rational singularities, a characterization of purity of local cohomology, a Thom--Sebastiani formula for the minimal exponent, and a linear combination formula for Bernstein--Sato polynomials of ideals. When the ambient variety is a complete intersection with rational singularities, we provide effective algorithms for these Bernstein--Sato polynomials and implement them in Macaulay2.    
\end{abstract}

\maketitle

\section{Introduction}

Let $X$ be a smooth complex variety and let $f\in \cO_X(X)$ be a nonzero regular function, defining the hypersurface $D=V(f)\subseteq X$. One of the most fundamental invariants attached to the singularities of $D$ is the \emph{Bernstein--Sato polynomial} $b_f(s)$. While defined algebraically via the theory of $\cD$-modules, these polynomials encode a remarkable amount of geometry: the largest root is the negative of the log canonical threshold \cites{Lichtin,Kollar}, and more generally the roots are closely tied to the jumping numbers of the
multiplier ideals of $D$ \cite{ELSV}. Through the $V$-filtration of Kashiwara and Malgrange \cites{KashiwaraV,MalgrangeV}, the $b$-function sits at the heart of Saito's theory of mixed Hodge modules \cite{saito90}, governing the nearby and vanishing cycles of $f$ as well as the Hodge and pole-order filtrations on the local cohomology module $\mathcal{H}^1_f(\cO_X)$ (see \cites{HodgeIdealsQV, MPHodgeFilt}).

A finer invariant is the \emph{microlocal $b$-function} $\widetilde{b}_f(s)$ of Saito \cite{SaitoMicrolocal}, obtained by allowing the formal inverse $\partial_t^{-1}$ in the functional equation via partial microlocalization. The negative of its largest root is the \emph{minimal exponent} $\widetilde{\alpha}_f$, which sharpens the log canonical threshold and detects Hodge-theoretic properties of $D$. Strikingly, $D$
has rational singularities if and only if $\widetilde{\alpha}_f>1$ \cite{saito93}. More recently, it was shown that the minimal exponent completely describes higher rational and higher Du Bois singularities of $D$ (see \cite{MPhigher}*{Thm. E}, \cite{FL}*{App. by M. Saito}, and \cites{MOPW,JKSY}).

The extension of this theory to singular ambient varieties $X$ is not straightforward, as $\cO_X$ no longer underlies a $\cD$-module, and the naive generalization of the Bernstein--Sato functional equation loses its Hodge-theoretic grounding. A replacement, introduced by Dirks \cite{DirksMultiplier}, is to use the intersection cohomology $\cD$-module ${\rm IC}_X$ in place of $\cO_X$. More precisely, given a nonzerodivisor $f \in \cO_X(X)$, the Bernstein--Sato polynomial $b_{(X,f)}(s)$ is defined through a functional equation via the Hodge module structure on ${\rm IC}_X$, written ${\rm IC}^H_X$. Its roots are always negative rational numbers and are related to jumping numbers of the \emph{multiplier modules} of $X$ along $f$. 

The purpose of this paper is to introduce and study the \emph{microlocal
$b$-function} $\widetilde{b}_{(X,f)}(s)$ of a function $f$ on a possibly singular
variety $X$, extending the analogy with the smooth theory. We define it by applying Saito's partial microlocalization operation to the construction used in the definition of $b_{(X,f)}(s)$. Importantly, $\widetilde{b}_{(X,f)}(s)$ can differ substantially from $b_{(X,f)}(s)$, and this discrepancy reflects the underlying geometry of the pair $(X,f)$. When $X$ is smooth, the two polynomials differ by exactly one factor: $\widetilde{b}_f(s)=b_f(s)/(s+1)$. In the singular setting, the microlocal $b$-function can strictly divide even the reduced polynomial $b_{(X,f)}(s)/(s+1)$ (see, for instance, Example \ref{eg:sumOfsquares}). This phenomenon necessitates the independent study of $\widetilde{b}_{(X,f)}(s)$, from which we develop the minimal exponent $\widetilde{\alpha}(X,f)$ and its associated Hodge-theoretic invariants.

The singular ambient setting also introduces subtleties of another kind. One new feature is that the conormal geometry of ${\rm IC}_X$, encoded in the characteristic variety, interacts with the hypersurface $D$ in ways that cannot occur when $X$ is smooth. These interactions provide obstructions that require a more refined analysis. Notably, $\widetilde{b}_{(X,f)}(s)$ can equal $1$ even when $D$ is singular, in stark contrast with the smooth ambient setting. In another direction, the presence of $f$-torsion in the characteristic variety can weaken the tight connection between the minimal exponent and the Hodge structure on local cohomology.

We now describe our main results. Throughout, we work in the following local setting: $X$ is a reduced and irreducible complex variety, embedded as a closed subvariety of a smooth variety $Y$, and $f\in \cO_Y(Y)$ restricts to a noninvertible function on $X$. We write $D=X\cap V(f)$ and $c=\dim Y-\dim X$. The precise conventions are fixed in Section \ref{s:setup}. The microlocal $b$-function $\widetilde{b}_{(X,f)}(s)$ is defined in Section \ref{sec:microBdef} as the minimal polynomial of the action of $s=-\partial_tt$ on a module obtained by applying Saito's partial microlocalization to the graph embedding of ${\rm IC}_X^H$ along $f$. The minimal exponent $\widetilde{\alpha}(X,f)$ is the negative of its largest root. When $X$ is smooth, these recover the classical microlocal $b$-function and minimal exponent of Saito \cite{SaitoMicrolocal}.

As mentioned above, when $X$ is smooth, $D$ has rational singularities if and only if $\widetilde{\alpha}_f>1$ \cite{saito93}. We generalize as follows.

\begin{thmx}[Theorem \ref{thm:ratSing}]\label{thmx:ratSing}
Using notation as above, assume that $X$ has rational singularities. Then $D = X \cap V(f)$ has rational singularities if and only if
\[
\widetilde\alpha(X,f) > 1.
\]
\end{thmx}

\noindent In particular, the minimal exponent $\widetilde{\alpha}(X,f)$ is the obstruction to passing rational singularities from an ambient $X$ to a divisor $D$. Theorem \ref{thmx:ratSing} is a consequence of a more general result, Corollary \ref{cor:GRadjunt}, which relates the Grauert--Riemenschneider sheaf of $D$ to that of $X$, provided that $\widetilde\alpha(X,f) > 1$ and a certain weighted Hodge piece of $\cH^c_X(\cO_Y)$ has no $f$-torsion. See Remark \ref{rmk-HigherPossibilities} for a discussion of higher Du Bois and rational singularities in relation to the minimal exponent. 

\begin{rmk} When $X$ is smooth, $\widetilde{\alpha}_f \geq 1$ is equivalent to $D$ having log canonical (or Du Bois) singularities \cite{SaitoHFilt}*{Thm. 0.5}. The analogue in the singular ambient setting is \cite{DirksMultiplier}*{Cor. 1.6}, which says that if $X$ has rational singularities and $\widetilde{\alpha}(X,f) \geq 1$, then $D$ has Du Bois singularities.
\end{rmk}

When $X$ is smooth, the minimal exponent controls the comparison between three important filtrations on the local cohomology module $\mathcal{H}^1_f(\cO_X)$: (1) the filtration by order of pole, (2) the Hodge filtration $F_{\bullet}$, and (3) the weight filtration $W_{\bullet}$. We recall the following special cases:
\begin{enumerate}
\item (log canonical singularities) $\widetilde{\alpha}_f\geq 1$ if and only if $F_0(\mathcal{H}^1_f(\cO_X))=\cO_Xf^{-1}/\cO_X$,

\item (rational singularities) $\widetilde{\alpha}_f> 1$ if and only if $F_0(W_{d_X+1}(\mathcal{H}^1_f(\cO_X)))=\cO_Xf^{-1}/\cO_X$.
\end{enumerate}
For (1), see \cite{HodgeIdealsQV} and references therein; for (2), see \cites{saito93, olano23}.

We generalize these results to the singular ambient setting, by replacing $\cO_X$ with ${\rm IC}_X$.

\begin{thmx}\label{thmx:minExpLC}
If $\widetilde\alpha(X,f) \geq 1$, then
\[
F_c\bigl(\mathcal{H}^1_f({\rm IC}_X^H)\bigr) = F_c({\rm IC}_X)f^{-1}/F_c({\rm IC}_X).
\]
If $\widetilde\alpha(X,f) > 1$, then 
\[
F_c\bigl(W_{d_X+1}(\mathcal{H}^1_f({\rm IC}_X^H)\bigr)) = F_c({\rm IC}_X)f^{-1}/F_c({\rm IC}_X).
\]
The converses hold if $\operatorname{gr}^F({\rm IC}^H_X)$ has no $f$-torsion.
\end{thmx}

\noindent The module $\operatorname{gr}^F({\rm IC}^H_X)$ has no $f$-torsion if the characteristic variety of ${\rm IC}_X$ has no component of the form $T^*_Z Y$ with $Z\subseteq V(f)$. When $X$ is smooth, $\operatorname{gr}^F({\rm IC}^H_X)$ is the coordinate ring of the zero section of the cotangent bundle, and hence has no $f$-torsion. Thus, we recover the known behavior for $\widetilde{\alpha}\geq 1$ (resp. $\widetilde{\alpha}> 1$) in the smooth ambient case. Theorem \ref{thmx:minExpLC} is a simplification of Theorem \ref{thm:minExpLC}, which has a weaker hypothesis on $\operatorname{gr}^F({\rm IC}^H_X)$ for the converse. These results are consequences of a general statement, namely Theorem \ref{thm:naiveOrd}, which relates containment of the Hodge, weight, and order filtrations to behavior of the (microlocal) $V$-filtration on the graph embedding of ${\rm IC}_X^H$. We remark that the characteristic variety of ${\rm IC}_X$ can be computed in Macaulay2 \cite{M2}, using the functions $\mathtt{IHmodule}$ and $\mathtt{characteristicIdeal}$. 

We recall that $X$ is a rational homology manifold if the natural map ${\rm IC}_X\to \mathbf{R}\Gamma_X(\cO_Y)[c]$ is a quasi-isomorphism. The \emph{Hodge rational homology level} of $X$, written ${\rm HRH}(X)$ and investigated in \cites{DOR1, ParkPopa}, measures how close $X$ is to being a rational homology manifold (see Section \ref{sec:LCpre} for more details). We show that, when $\widetilde{b}_{(X,f)}(s)$ has no integer roots, the Hodge rational homology level cannot decrease when passing from $X$ to $D$.

The following generalizes the fact that, if $X$ is smooth and $\widetilde{b}_f(s)$ has no integer roots, then $D\subseteq X$ is a rational homology manifold \cite{Torrelli}*{Thm. 1.2}, \cite{DimcaSaito}*{Thm. 3(c)}.

\begin{thmx}[Corollary \ref{cor-HRHDivisor}]\label{thmx:HRH}
If $\widetilde{b}_{(X,f)}(s)$ has no integer roots, then 
\[
{\rm HRH}(D) \;\geq\; {\rm HRH}(X).
\]
\end{thmx}

\noindent Theorem \ref{thmx:HRH} is in fact a consequence of a stronger purity statement, which we now state. Recall that when $X$ is smooth, $D$ is a rational homology manifold if and only if $\cH^1_f(\cO_X)$ is a pure Hodge module, of weight $d_X+1$. We generalize as follows.

\begin{thmx}[Theorem \ref{prop:microAndLocalization}]\label{thmx:purity}
The module $\mathcal{H}^1_f({\rm IC}_X^H)$ underlies a pure Hodge module, of weight $d_X+ 1$, if and only if $\widetilde{b}_{(X,f)}(s)$ has no integer roots.
\end{thmx}

\noindent This result points to the reason why Theorem \ref{thmx:HRH} provides a sufficient, but not necessary condition for ${\rm HRH}(D) \geq {\rm HRH}(X)$. Indeed, purity of $\mathcal{H}^1_f({\rm IC}_X^H)$ is not the same as purity of $\mathcal{H}^{c+1}_D(\cO_Y)$. Rather, the modules are related by a short exact sequence (see (\ref{e:GSS})). We prove Theorem \ref{thmx:purity} by relating purity of $\mathcal{H}^1_f({\rm IC}_X^H)$ to triviality of vanishing cycles of ${\rm IC}_X$ along $f$.  

In another direction, by taking the greatest common divisor over all open neighborhoods of a point $x\in X$, we define the local version of the (microlocal) Bernstein--Sato polynomial $b_{(X,f),x}(s)$ and $\widetilde{b}_{(X,f),x}(s)$. Similarly, we define the local minimal exponent $\widetilde{\alpha}_x(X,f)$. We have the following Thom--Sebastiani property for the minimal exponent, which generalizes the corresponding result for smooth ambient varieties \cite{SaitoMicrolocal}*{Thm. 0.7} (see  \cite{HodgeIdealsQV}*{Example 6.7}).

\begin{thmx} \label{thm-TS} For $i=1,2$, let $X_i$ be an irreducible subvariety of a smooth variety $Y_i$. Let $f_i \in \cO_{Y_i}(Y_i)$ restrict to a nonzero noninvertible function on $X_i$ for $i=1,2$. Let $x_i \in V(f_i) \cap X_i =: D_i$. Then
\[ \widetilde{\alpha}_{(x_1,x_2)}(X_1\times X_2, f_1+f_2) = \widetilde{\alpha}_{x_1}(X_1, f_1) + \widetilde{\alpha}_{x_2}(X_2, f_2).\]
\end{thmx}

\noindent In particular, we can have $\widetilde{\alpha}_{(x_1,x_2)}(X_1\times X_2, f_1+f_2)>1$ even if neither $\widetilde{\alpha}_{x_1}(X_1, f_1)>1$ nor $\widetilde{\alpha}_{x_2}(X_2, f_2)>1$.

Next, we relate Bernstein--Sato polynomials of ideals to microlocal $b$-functions of general linear combinations of the generators of the ideal. Following \cite{DirksMultiplier}, we write $b_{(X,\fra)}(s)$ for the $b$-function associated to an ideal $\fra\subseteq \cO_Y$.

The following is a generalization of \cite{mus22}*{Theorem 1.1} to singular ambient varieties. \Mustata's proof uses the fact that the microlocal $b$-function is the reduced $b$-function, whereas our result relies on a more general computation in the spirit of \cite{CDMO}*{Prop. 3.4}.

\begin{thmx}\label{thmx:linearCombo}
Let $X \subseteq Y$ be as above and let $\mathfrak{a} = (f_1, \ldots, f_r) \subseteq \cO_Y(Y)$ restrict to a nonzero ideal on $X$. Set $h = \sum_{i=1}^r y_i f_i \in \cO_{Y \times \A^r_y}$. Then
\[
b_{(X,\mathfrak{a})}(s) \;=\; \widetilde{b}_{(X \times \A^r_y,\, h)}(s).
\]
\end{thmx}

\noindent We use this to facilitate calculation of $b$-functions for ideals, using algorithms for microlocal $b$-functions. Theorem \ref{thmx:linearCombo} is proven by relating ${\rm IC}_{X\times \A^r_y}^H$ to the box product ${\rm IC}_X^H\boxtimes \Q^H_{\A^r}[r]$.

When $X$ is smooth, the divisor $D$ is smooth if and only if $\widetilde{b}_f(s)=1$. In Section \ref{sec:char}, we show that this equivalence fails in the singular ambient setting, and investigate necessary and sufficient conditions for $\widetilde{b}_{(X,f)}(s)=1$, using conormal geometry of ${\rm IC}_X$.

\subsection*{Effective Algorithms} One situation where we have a firm grasp on the first nonzero piece of the Hodge filtration on ${\rm IC}_X^H$ is when $X\subseteq Y$ is a complete intersection with rational singularities, defined by a regular sequence $g_1,g_2,\ldots,g_c$. In this case, the lowest Hodge piece $F_c({\rm IC}_X^H)$ is a cyclic $\cO_Y$-module, generated by the \v{C}ech class of $(g_1g_2\cdots g_c)^{-1}$ in $\cH^c_X(\cO_Y)$. This cyclic generator provides a bridge from the singular ambient setting to existing $\cD$-module algorithms via annihilators. In Section \ref{sec:Algs}, we exploit this fact to provide effective algorithms for 
\[
  b_{(X,f)}(s), \qquad \widetilde b_{(X,f)}(s), \qquad b_{(X,\fra)}(s),
\]
where $f$ restricts to a nonzero noninvertible function on $X$ and $\fra \subseteq \cO_Y$ restricts
to a nonzero nontrivial ideal on $X$. The third algorithm is obtained from the second via Theorem~\ref{thmx:linearCombo}. 

We implement these algorithms in the computer algebra system Macaulay2 \cite{M2}, and use this functionality to provide explicit examples throughout the document.

As mentioned above, $\widetilde b_{(X,f)}(s)$ is not generally equal to the ``reduced'' $b$-function $b_{(X,f)}(s)/(s+1)$. We provide an effective method to obtain $\widetilde b_{(X,f)}(s)$ from $b_{(X,f)}(s)$ via a sequence of intermediate polynomials obtained by iteratively dividing linear factors.

\subsection*{Organization}
Section~\ref{sec:Prelim} fixes notation and collects background on filtered $\cD$-modules, mixed Hodge modules, local cohomology, and the $V$-filtration. Section~\ref{sec:microBdef} introduces the microlocal $b$-function $\widetilde{b}_{(X,f)}(s)$, establishes its basic properties including the division relation comparing it to $b_{(X,f)}(s)$, and investigates conditions for $\widetilde{b}_{(X,f)}(s)=1$ via vanishing cycles and characteristic cycles. Section~\ref{sec:lcXf} uses $\widetilde{b}_{(X,f)}(s)$ to study $\mathcal{H}^1_f({\rm IC}_X^H)$, proving Theorems~\ref{thmx:HRH} and~\ref{thmx:purity}. Section~\ref{sec:minExp} defines the minimal exponent $\widetilde{\alpha}(X,f)$, compares the Hodge filtration on $\mathcal{H}^1_f({\rm IC}_X^H)$ to the naive pole-order filtration, and proves Theorems~\ref{thmx:ratSing} and~\ref{thmx:minExpLC}, as well as the Thom--Sebastiani formula of Theorem~\ref{thm-TS}. Section~\ref{sec:linearCombo} proves Theorem~\ref{thmx:linearCombo}. Finally, Section~\ref{sec:Algs} specializes to complete intersections with rational singularities, providing effective algorithms for $b_{(X,f)}(s)$, $\widetilde{b}_{(X,f)}(s)$, and $b_{(X,\fra)}(s)$, with Macaulay2 implementations and explicit examples.

\section{Preliminaries}\label{sec:Prelim} 

\subsection{The basic set-up}\label{s:setup} Throughout the paper, we adopt the following conventions. Let $X$ be a reduced and irreducible complex variety of dimension $d_X$. We embed $X$ into a smooth ambient variety $Y$ of dimension $d_Y$, and we write $c=d_Y-d_X$. We fix $f\in \cO_Y(Y)$ such that $f$ restricts to a nonzero noninvertible function on $X$. We define
\begin{equation}
D:=V(f)\cap X\subseteq X,    
\end{equation}
which is a Cartier divisor on $X$, of dimension $d_D=d_X-1$.

\subsection{Filtered D-modules and mixed Hodge modules}

Let $\cD=\cD_Y$ be the sheaf of algebraic differential operators on $Y$. Unless otherwise noted, our convention is that $\cD$-module means left $\cD$-module. We write $F_{\bullet}(\cD)$ for the increasing filtration by order of differential operator, whose associated graded object $\operatorname{gr}^F(\cD)$ satisfies $ \operatorname{gr}^F(\cD)\cong \operatorname{Sym}(\mathcal{T}_Y)$.

We consider mixed Hodge modules \cite{saito90} on $Y$, which consist of the following data
\begin{equation}\label{e:MHMdata}
M=(\mathcal{M},F_{\bullet},W_{\bullet}, (\mathcal{K},W_{\bullet}),\alpha),    
\end{equation}
where $\mathcal{M}$ is a regular holonomic $\cD$-module, referred to as the $\cD$-module underlying $M$. The Hodge filtration $F_{\bullet}(\mathcal{M})$ is a good filtration, meaning it is an exhaustive increasing $\Z$-indexed filtration by coherent $\cO_Y$-modules, compatible with $F_{\bullet}(\cD)$, and $\operatorname{gr}^F(\mathcal{M})$ is a coherent $\operatorname{gr}^F(\cD)$-module. The weight filtration $W_{\bullet}(\mathcal{M})$ is a finite exhaustive increasing $\Z$-indexed filtration by $\cD$-modules. Here, $\mathcal{K}$ is a $\Q$-constructible perverse sheaf on $Y^{\operatorname{an}}$, the complex analytic space corresponding to $Y$, and $W_{\bullet}$ is a finite filtration on $\mathcal{K}$, with a filtered isomorphism
\[
\alpha:(\mathcal{K},W_{\bullet})\otimes_{\Q}\C \to \operatorname{DR}_{Y^{\operatorname{an}}}(\cM^{\operatorname{an}},W_{\bullet}^{\operatorname{an}}),
\]
where $\cM^{\operatorname{an}}$ is the $\cD_{Y^{\operatorname{an}}}$-module associated to $\cM$ (see \cite{htt}*{Section 4.7}), $W_{\bullet}^{\operatorname{an}}$ is the corresponding filtration by $\cD_{Y^{\operatorname{an}}}$-modules, and $\operatorname{DR}_{Y^{\operatorname{an}}}$ is the analytic de Rham functor \cite{htt}*{Section 4.2}. In the sequel, we write ${\rm RH}(\cM,W_{\bullet})=\operatorname{DR}_{Y^{\operatorname{an}}}(\cM^{\operatorname{an}},W_{\bullet}^{\operatorname{an}})$.

The (algebraic) de Rham complex of $\cM$ on smooth $Y$ is
\[ {\rm DR}(\cM) = \left(\cM \xrightarrow[]{\nabla} \Omega_Y^1 \otimes_{\cO_Y} \cM \xrightarrow[]{\nabla} \dots \xrightarrow[]{\nabla} \omega_Y \otimes_{\cO_Y} \cM \right),\]
where $\nabla \colon \cM \to \Omega_Y^1 \otimes_{\cO_Y} \cM$ is the connection on $\cM$ determined by the $\cD_Y$-module structure. Importantly, the differentials of this complex are \emph{not} $\cO_Y$-linear. If $(\cM,F)$ is a filtered $\cD_Y$-module, then the de Rham complex is also filtered
\[ F_p {\rm DR}(\cM) = \left( F_p \cM \xrightarrow[]{\nabla} \Omega_Y^1 \otimes_{\cO_Y} F_{p+1} \cM \xrightarrow[]{\nabla} \dots \xrightarrow[]{\nabla} \omega_Y \otimes_{\cO_Y} F_{p+d_Y} \cM\right),\]
where the choice of index is made for compatibility with the natural definition for right $\cD_Y$-modules (and for compatibility with the direct image functor). By the Leibniz rule, we see that, for any $p\in \Z$, the complex
\[ {\rm Gr}^F_p {\rm DR}(\cM) = \left( {\rm Gr}^F_p \cM \xrightarrow[]{\nabla} \Omega_Y^1 \otimes_{\cO_Y} {\rm Gr}^F_{p+1} \cM \xrightarrow[]{\nabla} \dots \xrightarrow[]{\nabla} \omega_Y \otimes_{\cO_Y} {\rm Gr}^F_{p+d_Y} \cM\right)\]
has $\cO_Y$-linear differentials, hence determines an object in $D^b_{\rm coh}(\cO_Y)$. 

We have the following result (which follows easily by definition) that translates between the lowest filtered piece and the first nontrivial ${\rm Gr}^F_p {\rm DR}(\cM)$:
\begin{lem} \label{lem-lowestHodge} Let $(\cM,F)$ be a filtered $\cD_Y$-module. Then we have
\[ \min\{\sigma \mid F_\sigma \cM \neq 0\} = d_Y + \min\{p \mid {\rm Gr}^F_p {\rm DR}(\cM) \neq 0\}.\]
We let $p(\cM,F) = \min\{p \mid {\rm Gr}^F_p {\rm DR}(\cM) \neq 0\}$.
\end{lem}

We denote by $\operatorname{MHM}(Y)$ the abelian category of mixed Hodge modules on $Y$. Morphisms $\phi:M\to N$ are strict with respect to the Hodge and weight filtrations, in the sense that
\begin{equation}
\phi(F_k(\mathcal{M}))=\phi(\mathcal{M})\cap F_k(\mathcal{N}),\quad \textnormal{and}\quad  \phi(W_k(\mathcal{M}))=\phi(\mathcal{M})\cap W_k(\mathcal{N}),   
\end{equation}
for all $k\in \Z$, where $(\mathcal{M},F_{\bullet},W_{\bullet})$ and $(\mathcal{N},F_{\bullet},W_{\bullet})$ are the filtered $\cD_Y$-modules underlying $M$ and $N$, respectively. In particular, $\operatorname{Gr}^F_k$ and $\operatorname{Gr}^W_k$ are exact functors.

We say that $M$ as in (\ref{e:MHMdata}) is pure of weight $w\in \Z$ if ${\rm Gr}^W_w \mathcal{M}=\mathcal{M}$. The trivial Hodge module $\Q_Y^H[d_Y]$ has underlying $\cD$-module $\cO_Y$, and the Hodge and weight filtrations satisfy ${\rm Gr}^F_0 \cO_Y = \cO_Y$ and ${\rm Gr}^W_{d_Y} \cO_Y = \cO_Y$. Hence, $\Q_Y^H[d_Y]$ is pure of weight $d_Y$. Given an equidimensional subvariety $Z\subseteq Y$ of dimension $d_Z$, we consider the intersection cohomology Hodge module ${\rm IC}^H_Z$, which is pure of weight $d_Z$, and has underlying $\cD$-module ${\rm IC}_Z$, the intersection cohomology $\cD$-module of Brylinski--Kashiwara. For instance, $\Q_Y^H[d_Y]={\rm IC}^H_Y$, though this equality is possible for singular varieties too. Such varieties are called \emph{rational homology manifolds}.

\begin{eg} Let $(\cM,F_\bullet,W_\bullet)$ be a bi-filtered $\cD_Y$-module underlying a mixed Hodge module $M \in {\rm MHM}(Y)$. For $k\in \Z$, the Tate twist $M(k) \in {\rm MHM}(Y)$ has the same underlying $\cD_Y$-module, but the filtrations are shifted: namely, we have
\[ W_\bullet (\cM(k)) = W_{\bullet+2k} \cM, \quad F_\bullet (\cM(k)) = F_{\bullet-k} \cM.\]
In particular, if $M$ is pure of weight $w$, then $M(k)$ is pure of weight $w-2k$.

For example, for $k\in \Z$, let ${\rm IC}^H_Z(k)$ be the $k$-th Tate twist of ${\rm IC}^H_Z$, which is pure of weight $d_Z-2k$, and whose Hodge filtration $F_{\bullet}$ satisfies  
\begin{equation}
F_{d_Y-d_Z+k-1}({\rm IC}_Z)=0,\quad F_{d_Y-d_Z+k}({\rm IC}_Z)\neq 0.   
\end{equation}
\end{eg}

A mixed Hodge module $M$ has weights $\geq k$ if $W_j(\mathcal{M})=0$ for $j<k$. A complex of mixed Hodge modules $M^\bullet$ has weights $\geq k$ if, for all $i\in \Z$, the cohomology $\cH^i M^\bullet$ has weights $\geq k+i$.

There is the exact duality functor
\begin{equation}
\mathbf{D}_Y:\operatorname{MHM}(Y)^{\operatorname{op}}\to \operatorname{MHM}(Y),
\end{equation}
which lifts the holonomic duality functor of $\cD_Y$-modules \cite{htt}*{Section 2.6}.

In fact, Saito's theory satisfies the full six functor formalism (compatibly with that on $\Q$-perverse sheaves and holonomic $\cD$-modules). Indeed, using local embeddings into smooth varieties, one can define the category ${\rm MHM}(X)$ even for a singular variety $X$. Importantly, for the six functor formalism, one must then consider the bounded derived categories $D^b({\rm MHM}(X))$.

Although these functors make sense for any morphism of varieties $f\colon X_1 \to X_2$, in this paper we only need to make use of the direct and inverse image functors for a closed embedding $i\colon X \to Y$.

The functor $i_* \colon D^b({\rm MHM}(X)) \to D^b({\rm MHM}(Y))$ is defined as an exact functor at the abelian category level
\[ i_* \colon {\rm MHM}(X) \to {\rm MHM}(Y).\]

The essential image is the full subcategory ${\rm MHM}_X(Y)$ consisting of mixed Hodge modules whose support (of underlying $\Q$-perverse sheaf, say) is set-theoretically contained in $X$. In fact, this is one definition of the category ${\rm MHM}(X)$ when $X$ is singular (if $Y$ is smooth). As $i$ is proper, we have that $i_* = i_!$, in other words, $i_*$ commutes with the duality functor.

Formally, the inverse image functors $i^!, i^* \colon D^b({\rm MHM}(Y)) \to D^b({\rm MHM}(X))$ are defined by the expected adjunction properties: $i^!$ is the right adjoint to $i_! = i_*$ and $i^*$ is the left adjoint to $i_*$. It turns out, in the closed embedding setting, that the natural morphisms
\[ i^* i_* \to {\rm id} \text{ and } {\rm id} \to i^! i_*\]
are isomorphisms of endofunctors on $D^b({\rm MHM}(X))$.

As in any six functor formalism, the inverse image functors are related by duality: $\mathbf D \circ i^* \circ \mathbf D \cong i^!$. Moreover, if $M^\bullet \in D^b({\rm MHM}(Y))$ has weights $\geq w$, then $i^! M^\bullet$ also has weights $\geq w$: in particular,
\[ \cH^j i^! M^\bullet \text{ has weights } \geq w+j.\]

Now, let $i \colon X \hookrightarrow Y$ be a closed embedding as in our basic set-up (see Section \ref{s:setup}). Applying $i_* i^!(-)$ to the constant Hodge module $\Q_Y^H[d_Y]$ yields the local cohomology mixed Hodge modules (see also Section \ref{sec:LCpre}). At the $\cD_Y$-module level, these are defined as the cohomology modules of the \v{C}ech complex
\[ \cO_Y \to \bigoplus_{i=1}^r \cO_Y(*D_i) \to \bigoplus_{1\leq i<j\leq r} \cO_Y(*(D_i+D_j))\to \dots \to \cO_Y(*D),\]
where $I \subseteq \cO_Y$ is the ideal sheaf of $X$, locally generated by $g_1,\dots, g_r \in \cO_Y$ and $D_i = V(g_i)$. We write $D = \sum_{i=1}^r D_i$. Here, $\cO_Y(*D_i)$ denotes the localization of $\cO_Y$ along $D_i$. Saito's theory endows these $\cD_Y$-modules with a mixed Hodge module structure. 

If $i \colon D = \{f=0\} \to Y$ is the inclusion of a divisor into a smooth variety $Y$, then the restriction functors can be understood in terms of the nearby and vanishing cycles. These are important exact functors which are classically defined as sheaf theoretic generalizations of the Milnor fiber cohomology (see, for instance, \cite{dimca}).

Given $f\in \cO_Y$, there are the exact functors of nearby (resp. vanishing) cycles 
\[\psi_f \, (\text{resp. } \varphi_f) \colon {\rm MHM}(Y) \to {\rm MHM}(D)\]
and morphisms
\[ {\rm Var} \colon \varphi_f(M) \to \psi_f(M)(-1),\]
\[ {\rm can} \colon \psi_f(M) \to \varphi_f(M),\]
so that, given $M \in {\rm MHM}(Y)$, we have 
\[ i^!(M) = (\varphi_f(M) \xrightarrow[]{\rm Var} \psi_f(M)(-1)) , \quad i^*(M) =( \psi_f(M) \xrightarrow[]{\rm can} \varphi_f(M))[1]\]
in the derived category $D^b({\rm MHM}(D))$. When we introduce the $V$-filtration in Section \ref{s:microPrelim}, we will recall how these are computed at the filtered $\cD$-module level. 

We recall the operation $\boxtimes$ on mixed Hodge modules \cite{saito90}*{(4.4.2)} in our setting. Let $Y_1$ and $Y_2$ be smooth complex varieties. For $M\in \operatorname{MHM}(Y_1)$ there is an exact functor
\begin{equation}
M\boxtimes - :\operatorname{MHM}(Y_2)\to \operatorname{MHM}(Y_1\times Y_2).   
\end{equation}
Let $(\mathcal{M},F_{\bullet},W_{\bullet})$ be the bi-filtered $\cD_{Y_1}$-module underlying $M$. Given $N\in \operatorname{MHM}(Y_2)$ with underlying bi-filtered $\cD_{Y_2}$-module $(\mathcal{N},F_{\bullet},W_{\bullet})$, the mixed Hodge module $M\boxtimes N$ has underlying $\cD_{Y_1\times Y_2}$-module $\cM\boxtimes \mathcal{N}$ and filtrations
\begin{equation}
F_k(\cM\boxtimes \mathcal{N})=\sum_{i+j=k}F_i(\cM)\boxtimes F_j(\mathcal{N}),\quad  W_k(\cM\boxtimes \mathcal{N})=\sum_{i+j=k}W_i(\cM)\boxtimes W_j(\mathcal{N}). 
\end{equation}

\subsection{Local cohomology as a mixed Hodge module}\label{sec:LCpre} Let $Z\subseteq Y$ be a closed equidimensional subvariety of dimension $d_Z$, and write $i:Z\to Y$ for the corresponding closed immersion. The local cohomology sheaves $\mathcal{H}^j_Z(\cO_Y)$ are endowed with structures as mixed Hodge modules, via the identification
\begin{equation}
\mathcal{H}^j_Z(\cO_Y)=\mathcal{H}^j(i_{\ast}i^{!}(\Q^H_Y[d_Y])).    
\end{equation}
As such, each $\mathcal{H}^j_Z(\cO_Y)$ carries a Hodge filtration $F_{\bullet}$ and a weight filtration $W_{\bullet}$. We have
\[
\mathcal{H}^j_Z(\cO_Y)=0\textnormal{ for $j<d_Y-d_Z$ and }\quad \mathcal{H}^{d_Y-d_Z}_Z(\cO_Y)\neq 0.
\]
The module $\mathcal{H}^{d_Y-d_Z}_Z(\cO_Y)$ has weights $\geq 2d_Y-d_Z$, and
\begin{equation}
W_{2d_Y-d_Z}( \mathcal{H}^{d_Y-d_Z}_Z(\cO_Y) )={\rm IC}_Z^H(d_Z-d_Y).   
\end{equation}
For $j>d_Y-d_Z$, the module $\mathcal{H}^j_Z(\cO_Y)$ has weights $\geq d_Y+j+1$ (see \cite{ParkPopa}*{Proposition 6.4}). For all $j\geq 0$ the Hodge filtration on $\mathcal{H}^j_Z(\cO_Y)$ satisfies $F_{-1}(\mathcal{H}^j_Z(\cO_Y))=0$.

Local cohomology can be used to measure the invariant ${\rm HRH}(Z)$ of \cite{DOR1} (and the corresponding $D_m$ property of \cite{ParkPopa}). This invariant is defined via the comparison morphism
\[ \phi \colon \Q_Z^H[d_Z ] \to {\rm IC}_Z^H,\]
as
\[ \sup\{ k \in \Z_{\geq -1} \mid {\rm Gr}^F_{-k} {\rm DR}(\phi) \text{ is a quasi-isomorphism}\}.\]

This gives a partial measure of Poincar\'{e} duality on the singular cohomology of $Z$. Moreover, the condition ${\rm HRH}(Z) \geq 0$ is precisely the difference between Du Bois and rational singularities.

By duality, the definition is easily seen to be equivalent to the following:
\begin{equation} \label{eq-equivalentHRH} {\rm HRH}(Z) = \max\{k \in \Z \mid F_k W_{d_Y+c_Z} \cH^{c_Z}_Z(\cO_Y) \to F_k \cH^\bullet_Z(\cO_Y)[c_Z] \text{ is a quasi-isom.}\},\end{equation}
where $c_Z=d_Y-d_Z$. It is straightforward to see that ${\rm HRH}(Z)\geq -1$.

\subsection{The (microlocal) V-filtration}\label{s:microPrelim} In this subsection, we review the $V$-filtration of holonomic $\cD_Y$-modules, as well as its microlocal analogue due to Saito \cite{SaitoMicrolocal}. See also \cite{ThomSebastiani} for further details and proofs in the latter case.

Using notation as in Section \ref{s:setup}, let 
\begin{equation}
\Gamma \colon Y \to Y \times \A^1_t    
\end{equation} 
be the graph embedding of $Y$ along $f$. We define
\begin{equation}
\mathcal{R}:=\cD_{Y\times \A^1_t},\quad \widetilde{\mathcal{R}}:= \cD_{Y\times \A^1_t}[\partial_t^{-1}],   
\end{equation}
and define the $\Z$-indexed $V$-filtrations on these rings by
\begin{equation}
V^k\mathcal{R}:=\sum_{i-j\geq k, j \geq 0}\cD_Yt^i\partial_t^j,\quad V^k\widetilde{\mathcal{R}}:=\sum_{i-j\geq k}\cD_Yt^i\partial_t^j.    
\end{equation}
We let $\cM$ denote the filtered $\cD_Y$-module underlying a mixed Hodge module $M\in {\rm MHM}(Y)$ and consider the graph embedding direct image
\begin{equation}\label{eqn:graph}
\Gamma_+ \cM = \bigoplus_{k\geq 0} \cM \de_t^k\delta_f,    
\end{equation}
where $\delta_f$ is a formal symbol, and the action of $\mathcal{R}$ is described as follows. The actions of $\Theta\in \operatorname{Der}_{\C}(\cO_Y)$ and $t$ are given by
\begin{equation}
\Theta\cdot m\partial_t^j\delta=\Theta(m)\partial_t^j\delta-\Theta(f)m\partial_t^{j+1}\delta,\quad \quad \textnormal{and}\quad\quad t\cdot m\partial_t^j\delta = fm\partial_t^j\delta-jm\partial_t^{j-1}\delta,    
\end{equation}
where $m$ is in $\cM$. The actions of $\cO_Y$ and $\partial_t$ are straightforward from (\ref{eqn:graph}).

We also consider the \emph{partial microlocalization} 
\begin{equation}
\Gamma_+ \cM [\de_t^{-1}].    
\end{equation}
The modules $\Gamma_+\cM$ and $\Gamma_+ \cM [\de_t^{-1}]$ are endowed with Hodge filtrations, given by 
\[F_p \Gamma_+\cM = \bigoplus_{k \geq 0} F_{p-k-1} \cM \de_t^k\delta_f,\]
\[ F_p \Gamma_+ \cM [\de_t^{-1}] = \bigoplus_{k \in \Z} F_{p-k-1} \cM \de_t^k\delta_f.\]

When $M = {\rm IC}_X^H$ for $X$ as in our basic set-up (see Section \ref{s:setup}), we will use the short-hand
\begin{equation}
\cB_f = \Gamma_+ \cM \text{ and } \tcB_f = \Gamma_+\cM[\de_t^{-1}],    
\end{equation} 
where we remark that $F_{c-1} \cM = 0$ and $F_c \cM \neq 0$.


Going back to an arbitrary $M \in {\rm MHM}(Y)$, the module $\Gamma_+\cM$ has a $\Q$-indexed $V$-filtration $V^\bullet \Gamma_+ \cM$ along $t$, which is a decreasing exhaustive filtration by quasi-coherent $\cO_{Y\times \A^1_t}$-modules. See \cite{CDM} for a general introduction. The partial microlocalization module $\Gamma_+ \cM [\de_t^{-1}]$ carries a \emph{microlocal $V$-filtration} \cites{SaitoMicrolocal,ThomSebastiani} determined by the $V$-filtration on $\Gamma_+ \cM$ by the formula
\begin{equation} \label{e:microVdef}
V^\alpha \Gamma_+ \cM[\de_t^{-1}] = \sum_{k\in \Z} \de_t^{k} V^{\alpha+k}\Gamma_+ \cM.
\end{equation}

For any $\lambda \in \Q$, the subspace $V^\lambda \Gamma_+ \cM$ is coherent over $V^0 \cR$, by definition. It is easy to see, then, that $V^\lambda (\Gamma_+ \cM[\de_t^{-1}])$ is coherent over $V^0 \widetilde{\cR}$ (see also the characterization of $V^\bullet \Gamma_+ \cM$ on \cite{ThomSebastiani}*{Pg. 4}). The $V$-filtration on $\Gamma_+ \cM$ makes it a filtered $(\mathcal{R},V)$-module in the sense that
\[
V^k\mathcal{R}\cdot V^{\alpha}\Gamma_+ \cM \subseteq V^{\alpha+k}\Gamma_+ \cM \quad \textnormal{for all $\alpha\in \Q$ and $k\in \Z$.}
\]
Similarly, $(\Gamma_+ \cM[\de_t^{-1}],V)$ is a filtered $(\widetilde{\cR},V)$-module.

For $\alpha \in \Q$, we define
\begin{equation}
 V^{>\alpha}\Gamma_+ \cM =\bigcup_{\beta>\alpha} V^{\beta}\Gamma_+\cM ,\quad \operatorname{Gr}^{\alpha}_V(\Gamma_+ \cM)=V^{\alpha}\Gamma_+ \cM/V^{>\alpha}\Gamma_+ \cM.   
\end{equation}
We define $V^{>\alpha}(\Gamma_+ \cM[\de_t^{-1}])$ and $\operatorname{Gr}^{\alpha}_V(\Gamma_+ \cM[\de_t^{-1}])$ similarly. For $\alpha\in \Q$ and $p\in \Z$ we write
\[
F_pV^{\alpha}\Gamma_+ \cM=F_p\Gamma_+ \cM\cap V^{\alpha}\Gamma_+ \cM,\;\; F_pV^{>\alpha}\Gamma_+ \cM=F_p\Gamma_+ \cM\cap V^{>\alpha}\Gamma_+ \cM,\;\; \]
\[F_p\operatorname{Gr}^{\alpha}_V(\Gamma_+ \cM)=F_pV^{\alpha}\Gamma_+ \cM/F_pV^{>\alpha}\Gamma_+ \cM.    
\]
We define $F_p\operatorname{Gr}^{\alpha}_V(\Gamma_+ \cM[\de_t^{-1}])$ similarly. It follows from (\ref{e:microVdef}) that we have filtered isomorphisms
\begin{equation}\label{e:compareVs}
{\rm Gr}_V^\alpha(\Gamma_+ \cM,F) \cong {\rm Gr}_V^{\alpha}(\Gamma_+ \cM[\de_t^{-1}],F)\quad \textnormal{for all $\alpha<1$}.    
\end{equation}

The (microlocal) $V$-filtrations are related to nearby and vanishing cycles of $\cM$ along $f$. Set
\[ \psi_{f,\lambda}(\cM) = {\rm Gr}_V^{\lambda}(\Gamma_+ \cM) \text{ for all } \lambda \in (0,1],\]
\[ \varphi_{f,1}(\cM) = {\rm Gr}_V^0(\Gamma_+ \cM), \quad \varphi_{f,\lambda}(\cM) = \psi_{f,\lambda}(\cM) = {\rm Gr}_V^\lambda (\Gamma_+\cM) \text{ for } \lambda \in (0,1),\]
with induced Hodge filtrations
\[ F_p \psi_{f,\lambda}(\cM) = F_p {\rm Gr}_V^\lambda(\Gamma_+ \cM)  \text{ for } \lambda \in (0,1],\quad F_p \varphi_{f,1}(\cM) = F_{p+1} {\rm Gr}_V^0(\Gamma_+ \cM).
\]
The shift of filtration on the second object can be explained by the morphisms
\[ {\rm Var} \colon \varphi_{f,1}(\cM) \to \psi_{f,1}(\cM)(-1),\quad {\rm can} \colon \psi_{f,1}(\cM) \to \varphi_{f,1}(\cM)\]
being given (at the un-filtered level) by
\[ t \colon {\rm Gr}_V^0(\Gamma_+ \cM) \to {\rm Gr}_V^1(\Gamma_+ \cM),\quad -\de_t \colon {\rm Gr}_V^1(\Gamma_+ \cM) \to {\rm Gr}_V^0(\Gamma_+ \cM).\]

By (\ref{e:compareVs}), the $V$-filtration on $\Gamma_+ \cM[\de_t^{-1}]$ contains all of the information of the vanishing cycles (but does not see the unipotent nearby cycles). 

\begin{eg} \label{eg-RestrictPureHM} Assume that $M$ is pure of weight $w$ on $Y$. As mentioned above, $i_* i^! M$, at the filtered $\cD_Y$-module level, is defined by the (strict) filtered morphism
\[t \colon {\rm Gr}_V^0(\Gamma_+ \cM,F[-1]) \to {\rm Gr}_V^1(\Gamma_+ \cM,F[-1]),\]
where the shift in Hodge filtration on the first object is by definition and that on the second object is due to the Tate twist in the definition of the map ${\rm Var}$. 

For the weight filtration, recall that for any $\lambda \in \Q$ the module ${\rm Gr}_V^\lambda (\Gamma_+ \cM)$ carries the nilpotent endomorphism $N = s+\lambda = -\de_t t + \lambda$. The weight filtration on ${\rm Gr}_V^\lambda(\Gamma_+ \cM)$ is defined as the \emph{monodromy filtration} centered at $w+ \lceil \lambda\rceil = \widetilde{w}$. In other words, it is the unique increasing filtration $W_\bullet {\rm Gr}_V^\lambda(\Gamma_+\cM)$ such that $N W_\bullet \subseteq W_{\bullet-2}$ and for all $\ell \in \Z_{\geq 0}$, the map
\[ N^\ell \colon {\rm Gr}^W_{\widetilde{w}+\ell} {\rm Gr}_V^\lambda(\Gamma_+ \cM) \to {\rm Gr}^W_{\widetilde{w}-\ell} {\rm Gr}_V^\lambda(\Gamma_+ \cM)\]
is an isomorphism. With this filtration, the morphism
\[ t \colon ({\rm Gr}_V^0(\Gamma_+ \cM,F[-1]),W) \to ({\rm Gr}_V^1(\Gamma_+ \cM,F[-1]),W),\]
is bi-strict with respect to $(F,W)$, and induces the Hodge and weight filtrations on the corresponding cohomology mixed Hodge modules.
\end{eg}

\section{Microlocal V-filtration, b-function, and vanishing cycles}

In this section, we introduce the microlocal $b$-function $\widetilde{b}_{(X,f)}(s)$, develop its basic properties, and relate its roots to vanishing cycles of ${\rm IC}_X$. We adopt the basic set-up from Section \ref{s:setup}.

\subsection{The microlocal b-function}\label{sec:microBdef}

Using notation as in Section \ref{s:microPrelim}, we consider the pure Hodge module ${\rm IC}^H_X$ with underlying $\cD$-module $\mathcal{M}={\rm IC}_X$. Recall that $F_c \cM$ is the first nonzero level of Hodge filtration on $\cM$. We write
\begin{equation}\label{e:Gdef}
G^\bullet \cB_f = V^\bullet (\mathcal{R}) \cdot (F_{c} \cM \delta_f),\quad \textnormal{and}\quad G^\bullet \widetilde{\cB}_f = V^\bullet (\widetilde{\mathcal{R}}) \cdot (F_{c} \cM \delta_f).    \end{equation}
These are exhaustive filtrations, compatible with the $V$-filtrations on $\mathcal{R}$ resp. $\widetilde{\mathcal{R}}$. 

We recall that \cite{DirksMultiplier} defined the $b$-function of $(X,f)$, written $b_{(X,f)}(s)$, as the minimal polynomial of the action of $s=-\de_t t$ on ${\rm Gr}_G^0(\cB_f)$. Since $V(f) \cap X \neq \emptyset$, we have $(s+1) \mid b_{(X,f)}(s)$.

One can give a minimal polynomial interpretation of the ``reduced'' Bernstein--Sato polynomial $\overline{b}_{(X,f)}(s) = b_{(X,f)}(s)/(s+1)$.

\begin{lem}\label{l:reduced} In the notation above, $\overline{b}_{(X,f)}(s-1)$ is the minimal polynomial of the $s$ action on ${\rm Gr}_G^{-1}(\cB_f)$.
\end{lem}
\begin{proof} By construction,
\[ G^{-1} \cB_f = V^0 \mathcal{R} \cdot (F_c \cM \de_t \delta_f + F_c \cM \delta_f) = \cD_Y[s] \cdot (F_c \cM \de_t \delta_f) + G^0 \cB_f.\]

One sees immediately that the minimal polynomial of the $s$ action on ${\rm Gr}_G^{-1}(\cB_f)$ can equivalently be described as the monic polynomial of least degree $p(s)$ such that
\[p(s) F_c\cM \de_t \delta_f \subseteq G^0 \cB_f.\]
Multiplying through by $t$, we get
\[ p(s+1) t\de_t F_c\cM \subseteq t\cdot G^0 \cB_f = G^1 \cB_f,\]
and $t\de_t = -(s+1)$. Thus, we conclude that $b_{(X,f)}(s) \mid p(s+1)(s+1)$, and so $\overline{b}_{(X,f)}(s) \mid p(s+1)$. It is clear that if we start with the functional equation for $b_{(X,f)}(s) = (s+1) \overline{b}_{(X,f)}(s) = - t\de_t \overline{b}_{(X,f)}(s)$ and divide by $t$, we get the other divisibility, hence proving the claim.
\end{proof}

We now introduce the main definition of the paper, that of the \emph{microlocal $b$-function}:
\begin{defi} The microlocal Bernstein--Sato polynomial of $(X,f)$, written $\widetilde{b}_{(X,f)}(s)$, is the minimal polynomial for the action of $s = -\de_t t$ on ${\rm Gr}_G^0 (\widetilde{\cB}_f)$.
\end{defi} 

In order to develop basic properties, we establish some useful properties of the $G$-filtration. 

\begin{lem} \label{lem:microGrBasics} For any $k \in \Z$, let $V^\lambda {\rm Gr}_G^k(\tcB_f) = \frac{V^\lambda G^k \tcB_f}{V^\lambda G^{k+1}\tcB_f}$ be the induced filtration on ${\rm Gr}_G^k(\tcB_f)$. Then
\begin{enumerate} \item For all $i\in \Z$, the morphism
\[ \de_t^i \colon V^\lambda G^k \tcB_f \to  V^{\lambda-i} G^{k-i}\tcB_f\]
is an isomorphism.

\item For all $i\in \Z$, the morphisms
\[ \de_t^i \colon V^\lambda {\rm Gr}_G^k \tcB_f \to  V^{\lambda-i} {\rm Gr}_G^{k-i} \tcB_f, \quad \de_t^i \colon {\rm Gr}_V^\lambda {\rm Gr}_G^k \tcB_f \to  {\rm Gr}_V^{\lambda-i} {\rm Gr}_G^{k-i} \tcB_f\]
are isomorphisms.

\item For all $\lambda < 1$, the morphisms
\[ G^k {\rm Gr}_V^\lambda(\cB_f) \to G^k {\rm Gr}_V^\lambda(\tcB_f) , \quad {\rm Gr}_G^k {\rm Gr}_V^\lambda(\cB_f) \to {\rm Gr}_G^k {\rm Gr}_V^\lambda(\tcB_f)\]
are isomorphisms.
\end{enumerate}
\end{lem}
\begin{proof} The second claim follows immediately from the first, which is immediate from the fact that (see (\ref{e:microVdef}) and (\ref{e:Gdef}))
\[ \de_t^i V^\lambda \tcB_f \cong V^{\lambda-i} \tcB_f, \quad \de_t^i G^k \tcB_f \cong G^{k-i}\tcB_f.\]

For the third claim, the second map being an isomorphism will follow from the first one being an isomorphism, so we reduce to proving that
\begin{equation}\label{e:desiredGmap}
G^k {\rm Gr}_V^\lambda(\cB_f) \to G^k {\rm Gr}_V^\lambda(\tcB_f),
\end{equation}
is an isomorphism for all $\lambda < 1, k \in \Z$. We use repeatedly the containment $(F_c \cM \delta_f) \subseteq V^0 \cB_f \subseteq V^0 \tcB_f$, which holds because $F_c \cM \delta_f$ is the lowest Hodge piece of $\cB_f$ (using \cite{SaitoMHP}*{(3.2.1.3)}).

If $k \geq 1$, the above containment implies $G^1 \subseteq V^{1} \subseteq V^{>\lambda}$ (in either $\cB_f$ or $\tcB_f$), and so both the source and target of (\ref{e:desiredGmap}) are zero and there is nothing to prove. 

Now assume that $k \leq 0$. Unpacking (\ref{e:desiredGmap}), we are interested in the morphism
\[ \frac{G^k V^\lambda \cB_f}{G^{k+1} V^\lambda \cB_f + G^k V^{>\lambda}\cB_f} \to \frac{G^k V^\lambda \tcB_f}{G^{k+1} V^\lambda \tcB_f + G^k V^{>\lambda}\tcB_f}.\]

We need to prove
\begin{equation}\label{e:desSurj}
G^k V^\lambda \tcB_f = G^k V^\lambda \cB_f + G^{k+1} V^\lambda \tcB_f + G^k V^{>\lambda}\tcB_f,    
\end{equation} 
which gives surjectivity, and
\begin{equation}\label{e:desInj}
(G^k V^\lambda \cB_f) \cap (G^{k+1} V^\lambda \tcB_f + G^k V^{>\lambda}\tcB_f) = G^{k+1} V^\lambda \cB_f + G^k V^{>\lambda}\cB_f,    
\end{equation}
which gives injectivity of (\ref{e:desiredGmap}). Note that, in either equality, the containment $\supseteq$ is obvious.

To verify (\ref{e:desSurj}), let $P \cdot m\delta_f \in G^k V^\lambda \tcB_f$, so that $P \in V^k \widetilde{\cR}$. By collecting all summands in $P$ with negative powers of $\de_t$, we can write $P = Q + P_0$ where $Q \in V^k \cR$ and $P_0 \in \de_t^{-1} \cD_Y[t,\de_t^{-1}]$. Thus, we have
\[ P \cdot m\delta_f = Q\cdot m\delta_f + P_0 \cdot m\delta_f \in G^k V^\lambda \tcB_f,\]
and by construction, $P_0 m \delta_f \in G^1 V^{1}\tcB_f \subseteq G^k V^{>\lambda} \tcB_f$. In particular, $Q \cdot m\delta_f \in V^\lambda \tcB_f \cap \cB_f$, so $Q\cdot m\delta_f \in V^\lambda \cB_f$ (using $\lambda < 1$, see \cite{ThomSebastiani}*{Pg. 4}). It lies in $G^k \cB_f$, too, by construction, so the first equality has been established.

To verify (\ref{e:desInj}), we can rewrite it using what we just showed as
\[ (G^k V^\lambda \cB_f) \cap (G^{k+1} V^\lambda \cB_f + G^k V^{>\lambda}\tcB_f) = G^{k+1} V^\lambda \cB_f + G^k V^{>\lambda}\cB_f,\]
and so we can take $\eta + \mu \in (G^k V^\lambda \cB_f) \cap (G^{k+1} V^\lambda \cB_f + G^k V^{>\lambda}\tcB_f)$, which we want to prove lies in $G^{k+1} V^\lambda \cB_f + G^k V^{>\lambda}\cB_f$. Using the fact that $\eta \in G^{k+1}V^\lambda \cB_f \subseteq G^k V^\lambda \cB_f$, we see that $\mu$ also lies in $G^k V^\lambda \cB_f$. But by assumption, $\mu$ lies in $G^k V^{>\lambda}\tcB_f$, too, so we conclude that $\mu \in V^{>\lambda}\cB_f$ (again using that $\lambda < 1$), which proves the claim. 
\end{proof}

Using Lemma \ref{lem:microGrBasics}, we establish basic results about roots of the microlocal $b$-function.

\begin{lem}\label{lem:microRootBasics}
Using notation as above, we have the following.
\begin{enumerate}
    \item Given $k\geq 1$ and $\lambda \in \Q$ we have
\[ (s+\lambda)^k \mid \widetilde{b}_{(X,f)}(s) \text{ if and only if } (s+\lambda)^{k-1} {\rm Gr}_V^\lambda {\rm Gr}_G^0(\widetilde{\cB}_f) \neq 0.\]

\item Moreover, for any $\mu \in \Q$, we have that
\[ \widetilde{b}_{(X,f)}(-\mu-k)\neq 0 \text{ for all } k \in \Z \text{ if and only if } {\rm Gr}_V^\mu(\widetilde{\cB}_f) = 0,\]
which holds if and only if we have vanishing
\[ \varphi_{f,\exp(-2\pi i \mu)}({\rm IC}_X) = {\rm RH}(\varphi_{f,\{\mu\}}(\cM)) = 0,\]
where ${\rm RH}(-)$ is the Riemann-Hilbert functor and $\{\mu\} \in (0,1]$ is the fractional part of $\mu$.
\end{enumerate}
\end{lem}

\begin{proof} To see (1), note that the induced $V$-filtration on ${\rm Gr}_G^0(\widetilde{\cB}_f)$ is bounded above and below. Indeed, we know that $F_{c} \cM \delta_f \subseteq V^0 \widetilde{\cB}_f$ as it is the first nonzero piece of the Hodge filtration, so that $V^0 {\rm Gr}_G^0(\widetilde{\cB}_f) = {\rm Gr}_G^0(\widetilde{\cB}_f)$. Fix any $\mu \in \Q$. Then by coherence of $V^\mu \widetilde{\cB}_f$ over $V^0 \cD_{Y}[s,\de_t^{-1}]$ and exhaustiveness of the filtration $G^\bullet \widetilde{\cB}_f$, it is easy to see that there exists $k(\mu)$ such that $V^\mu \widetilde{\cB}_f \subseteq G^{k(\mu)} \widetilde{\cB}_f$. If $k(\mu) \geq 1$, then we see that $V^\mu {\rm Gr}_G^0(\widetilde{\cB}_f) = 0$. Otherwise, by applying $\de_t^{-b}$ for $b = 1-k(\mu)$, we get containment $\de_t^{-b} V^\mu \widetilde{\cB}_f = V^{\mu+b} \widetilde{\cB}_f \subseteq G^{k(\mu) + b} \widetilde{\cB}_f$. Thus $V^{\mu+b} {\rm Gr}_G^0(\widetilde{\cB}_f) = 0$.

At this point, the claim (1) follows from linear algebra. Indeed, if $\cV$ is a vector space with an operator $T$ and ${\rm Fil}^\bullet \cV$ is a decreasing discrete, bounded above and below $\Q$-indexed filtration which is stable by $T$, such that $T+\lambda$ is nilpotent on ${\rm Gr}_{\rm Fil}^\lambda \cV$, then $T$ admits a minimal polynomial $b_T(w)$ which splits completely over $\Q$ and satisfies
\[ (w+\lambda)^k \mid b_T(w) \iff (T+\lambda)^{k-1} {\rm Gr}_{\rm Fil}^{\lambda}(\cV) \neq 0.\]

If $k(\lambda)$ is minimal such that $(T+\lambda)^{k(\lambda)} {\rm Gr}_{\rm Fil}^\lambda(\cV) = 0$, the claim is that
\[ b_T(w) = \prod_{\lambda} (w+\lambda)^{k(\lambda)}.\]
It is easy to see that $b_T(w) \mid \prod_{\lambda}(w+\lambda)^{k(\lambda)}$ by finiteness of the filtration ${\rm Fil}^\bullet \cV$. For the other division relation, we use the B\'{e}zout identity: there exist $p(w),q(w) \in \C[w]$ such that
\begin{equation}\label{e:bezout} 
p(w) b_T(w) + q(w) (w+\lambda)^{k(\lambda)} = (w+\lambda)^{\ell},
\end{equation}
for some $\ell \leq k(\lambda)$. Suppose for contradiction that $\ell <k(\lambda)$. If we use (\ref{e:bezout}) for $w= T$ and apply to ${\rm Gr}_{\rm Fil}^\lambda(\cV)$, then the left hand side annihilates but the right-hand side is nonzero by definition of $k(\lambda)$, a contradiction. This completes the proof of (1).

Item (2) follows from (1). Indeed, we have
\[ \widetilde{b}_{(X,f)}(-\mu-k) \neq 0 \text{ for all } k \in \Z,\]
if and only if ${\rm Gr}_V^{\mu+k} {\rm Gr}_G^0(\widetilde{\cB}_f) = 0$ for all $k \in \Z$. By the isomorphism (see Lemma \ref{lem:microGrBasics}(2))
\[\de_t^k \colon {\rm Gr}_V^{\mu+k} {\rm Gr}_G^0(\widetilde{\cB}_f) \to {\rm Gr}_V^{\mu} {\rm Gr}_G^{-k}(\widetilde{\cB}_f), \]
this is true if and only if ${\rm Gr}_G^{-k} {\rm Gr}_V^{\mu}(\widetilde{\cB}_f) = 0$ for all $k\in \Z$. By exhaustiveness of the $G$-filtration and the fact that, for $p \gg 0$ we have $G^p {\rm Gr}_V^{\mu}(\widetilde{\cB}_f) = 0$, this is true if and only if ${\rm Gr}_V^{\mu}(\widetilde{\cB}_f)= 0$.

The last claim follows from the Riemann-Hilbert correspondence, using that
\[ {\rm Gr}_V^{\mu}(\widetilde{\cB}_f) \xrightarrow[]{\de_t^{\mu - \widehat{\mu}}} {\rm Gr}_V^{\widehat{\mu}}(\widetilde{\cB}_f) \cong {\rm Gr}_V^{\widehat{\mu}}(\cB_f),\]
where $\widehat{\mu} \in [0,1)$ is the fractional part of $\mu$ (hence $\mu - \widehat{\mu} \in \Z$). The right isomorphism is given by \eqref{e:compareVs}.
\end{proof}

Next, we prove a division relation comparing $\widetilde{b}_{(X,f)}(s)$ to $b_{(X,f)}(s)$. Recall that, when $X$ is smooth, we have $\widetilde{b}_f(s)=b_f(s)/(s+1)$ \cite{SaitoMicrolocal}*{Prop. 0.3}. In the singular ambient setting, these two polynomials differ by a factor with distinct integral roots.

\begin{prop}\label{prop-bFunctionDivide}
We have that $\widetilde{b}_{(X,f)}(s)$ divides $b_{(X,f)}(s)$. Moreover, there exist distinct positive integers $k_0 = 1, k_1,\ldots,k_m$ such that
\begin{equation}
\widetilde{b}_{(X,f)}(s)=\frac{b_{(X,f)}(s)}{(s+1)(s+k_1)\cdots (s+k_m)}.    
\end{equation}
\end{prop}

\begin{proof}
We first show that $\widetilde{b}_{(X,f)}(s)$ divides $b_{(X,f)}(s)$. In fact, we will show $\widetilde{b}_{(X,f)}(s) \mid \overline{b}_{(X,f)}(s)$, where $\overline{b}_{(X,f)}(s)$ is the reduced $b$-function of Lemma \ref{l:reduced}. After this verification, we obtain that we can always take $k_0 = 1$ as in the statement of this proposition. It is more convenient to show $\widetilde{b}_{(X,f)}(s-1) \mid \overline{b}_{(X,f)}(s-1)$. Indeed, the former is the minimal polynomial of the $s=-\de_t t$ action on ${\rm Gr}_G^{-1}(\tcB_f) \cong  \de_t {\rm Gr}_G^0(\tcB_f)$, and by Lemma \ref{l:reduced} we know that the latter is the minimal polynomial of the $s$ action on ${\rm Gr}_G^{-1}(\cB_f)$.

We have the containment $G^j \cB_f \subseteq G^j \tcB_f$ for all $j\in \Z$, hence there is an induced morphism
\begin{equation} \label{eq-comparisonMap} \pi \colon {\rm Gr}_G^{-1}(\cB_f) \to {\rm Gr}_G^{-1}(\tcB_f)\end{equation}
compatible with the operator $s$. So for divisibility, it suffices to prove that this map is surjective. So we consider $m \in F_c\cM$ and $P \in V^{-1}\tcR$, giving $P m \delta_f \in G^{-1} \tcB_f = V^{-1}\tcR \cdot (F_c \cM\delta_f)$. By definition of the $V$-filtration on $\tcR$, we can write $P = Q \de_t + P_0$ where $Q\in V^0 \cR$ and $P_0 \in \cD_Y[\de_t^{-1}]$. Then
\[ P m \delta_f = (Q \de_t m \delta_f) + (P_0 m \delta_f) \in G^{-1} \cB_f + G^0 \tcB_f,\]
which proves surjectivity of $\pi$, as $[Pm\delta_f] = \pi([Q\de_t m\delta_f])$.

Next, we prove the second assertion. Our goal is to understand the kernel of the map $\pi$ \eqref{eq-comparisonMap}. For this, let $\lambda\in \Q$ and let $j = \lfloor \lambda \rfloor$, so that $\lambda -j \in [0,1)$. Then we have the commutative diagram
\begin{equation} \label{eq-commDiagram} \begin{tikzcd}  {\rm Gr}_G^{-1}(\cB_f) \ar[r, "\pi"] \ar[d,"\de_t^j"] & {\rm Gr}_G^{-1}(\tcB_f)  \ar[d,"\de_t^j"] \\ {\rm Gr}_G^{-j-1}(\cB_f)  \ar[r] &  {\rm Gr}_G^{-j-1}(\tcB_f)  \end{tikzcd},\end{equation}
where the right vertical map is an isomorphism, by definition of the $G$-filtration.

We have the decompositions
\[ {\rm Gr}_G^{-1}(\cB_f) = \bigoplus_{\lambda \in \Q} {\rm Gr}_V^\lambda {\rm Gr}_G^{-1}(\cB_f),\]
\[ {\rm Gr}_G^{-1}(\tcB_f) = \bigoplus_{\lambda \in \Q} {\rm Gr}_V^\lambda {\rm Gr}_G^{-1}(\tcB_f)\]
where the right-hand side is the decomposition into generalized eigenvectors for $s = -\de_t t$. In particular, the right-hand side is a finite direct sum, as $s$ has a minimal polynomial. Note that the commutative diagram \eqref{eq-commDiagram} splits along these decompositions: specifically, for any $\lambda \in \Q$, we have the induced commutative diagram
\begin{equation} \label{eq-commDiagramLambda} \begin{tikzcd}  {\rm Gr}_V^\lambda{\rm Gr}_G^{-1}(\cB_f) \ar[r, "\operatorname{Gr}^{\lambda}_V(\pi)"] \ar[d,"\de_t^j"] & {\rm Gr}_V^\lambda{\rm Gr}_G^{-1}(\tcB_f)  \ar[d,"\de_t^j"] \\ {\rm Gr}_V^{\lambda-j}{\rm Gr}_G^{-j-1}(\cB_f)  \ar[r] &  {\rm Gr}_V^{\lambda-j}{\rm Gr}_G^{-j-1}(\tcB_f)  \end{tikzcd}. \end{equation}

The right vertical morphism is an isomorphism by Lemma \ref{lem:microGrBasics}(2), and the bottom horizontal morphism is an isomorphism by Lemma \ref{lem:microGrBasics}(3). Thus, the kernel of $\operatorname{Gr}^{\lambda}_V(\pi)$ can be identified with the kernel of the left vertical morphism. 

We have the morphism
\[ {\rm Gr}_G^{-1}(\cB_f) \xrightarrow[]{\de_t^j} {\rm Gr}_G^{-j-1} (\cB_f) \xrightarrow[]{t^j} {\rm Gr}_G^{-1}(\cB_f),\]
which can be identified (up to sign) with the endomorphism $t^j\de_t^j = (s+1)(s+2)\dots (s+j)$. This composition also splits along the decomposition into generalized eigenspaces: we have
\[ {\rm Gr}_V^\lambda {\rm Gr}_G^{-1}(\cB_f) \xrightarrow[]{\de_t^j} {\rm Gr}_V^{\lambda-j}{\rm Gr}_G^{-j-1} (\cB_f) \xrightarrow[]{t^j} {\rm Gr}_V^\lambda {\rm Gr}_G^{-1}(\cB_f),\]
which is identified (up to sign) with the endomorphism $\phi_j(s):=(s+1)\dots (s+j)$. Thus, 
\[ \ker(\de_t^j \colon{\rm Gr}_V^\lambda {\rm Gr}_G^{-1}(\cB_f) \to {\rm Gr}_V^{\lambda-j} {\rm Gr}_G^{-j-1}(\cB_f)) \subseteq \ker(\phi_j(s) \colon {\rm Gr}_V^\lambda {\rm Gr}_G^{-1}(\cB_f) \to {\rm Gr}_V^\lambda {\rm Gr}_G^{-1}(\cB_f)).\]

If $\lambda \notin \{1,\dots, j\}$ (in other words, by choice of $j$, if $\lambda \notin \Z_{>0}$), then since $s+\lambda$ acts nilpotently, we conclude $\ker(\de_t^j \colon{\rm Gr}_V^\lambda {\rm Gr}_G^{-1}(\cB_f) \to {\rm Gr}_V^{\lambda-j} {\rm Gr}_G^{-j-1}(\cB_f)) = 0$. If $\lambda \in \Z_{>0}$, then we conclude that $s+\lambda$ annihilates the kernel (meaning one doesn't need higher powers of $s+\lambda$). 

Thus, we have an isomorphism $(s+\lambda) {\rm Gr}_V^\lambda {\rm Gr}_G^{-1}(\cB_f) = (s+\lambda) {\rm Gr}_V^\lambda {\rm Gr}_G^{-1}(\tcB_f)$ for $\lambda \in \Z_{>0}$. This (combined with the surjectivity of $\pi$) says that the multiplicity of $(s+\lambda)$ as a factor of $\widetilde{b}_{(X,f)}(s)$ is at most one less than its multiplicity in $\overline{b}_{(X,f)}(s)$, which proves that there are no repeats among the $k_1,\dots, k_m$ in the proposition statement, as claimed.
\end{proof}

\begin{rmk} In \cite{SaitoMicrolocal}*{Prop. 0.3}, Saito shows that if $X$ is smooth, then we have equality
\begin{equation}\label{e:miceqred}
\widetilde{b}_{(X,f)}(s) = \overline{b}_{(X,f)}(s).
\end{equation}
The proof relies on an explicit computation involving the action of $\cD_{Y\times \A^1_t}[\de_t^{-1}]$ on $1\cdot \delta_f \in \tcB_f$. Importantly, $1$ is annihilated by any differential operator $P$ of the form $P = \sum_{|\nu| > 0} a_\nu \de_y^\nu \in \cD_Y$. However, in the singular ambient setting, a similar computation does not suffice. In fact, we provide two counterexamples to the equality (\ref{e:miceqred}) in the following example. 
\end{rmk}

\begin{eg}\label{eg:sumOfsquares}
In the following, we set $X=V(x_1^2+x_2^2+x_3^2+x_4^2+x_5^2+x_6^2)\subseteq \A^6$. 
\begin{enumerate}
\item If $f=x_1$, then a calculation in Macaulay2 (see Section \ref{sec:Algs}) gives
\[
b_{(X,f)}(s)=(s+1)(s+4),\quad \textnormal{and}\quad \widetilde{b}_{(X,f)}(s)=1.
\]
In particular, $\widetilde{b}_{(X,f)}(s)$ is not equal to $\overline{b}_{(X,f)}(s)$.

\item If $f=x_1x_2+x_3x_4+x_5x_6$ then
\[
b_{(X,f)}(s)=(s+1)(s+2)\left(s+\frac{3}{2}\right)\quad \textnormal{and}\quad \widetilde{b}_{(X,f)}(s)=\left(s+\frac{3}{2}\right).
\]
Again, $\widetilde{b}_{(X,f)}(s)$ is not equal to $\overline{b}_{(X,f)}(s)$. 

\item On the other hand, if $f=x_1^2-1$, then 
\[
b_{(X,f)}(s)=(s+1)\quad \textnormal{and}\quad \widetilde{b}_{(X,f)}(s)=1.
\]
Thus, it is possible that $\widetilde{b}_{(X,f)}(s)=\overline{b}_{(X,f)}(s)$, even if $\widetilde{b}_{(X,h)}(s)\neq \overline{b}_{(X,h)}(s)$ for some $h\in \cO_Y(Y)$. Hence, the integers in Proposition \ref{prop-bFunctionDivide} depend on the pair $(X,f)$.
\end{enumerate}
We will revisit these examples throughout the document.
\end{eg}

It would be interesting to characterize varieties $X\subseteq Y$ for which $\widetilde{b}_{(X,f)}(s)=\overline{b}_{(X,f)}(s)$ for all $f\in \cO_Y(Y)$. As a first step, we pose the following question.

\begin{question}
Suppose $X\subseteq Y$ is a rational homology manifold. Do we have $\widetilde{b}_{(X,f)}(s)=\overline{b}_{(X,f)}(s)$ for all $f$ satisfying the basic set-up in Section \ref{s:setup}?    
\end{question}

\noindent See the following examples for instances when $X$ is a rational homology manifold: Example \ref{eg:A1}, Example \ref{eg:E8}, Example \ref{eg:HRHbound}(2).

\begin{eg}
A natural question in the case $X$ is a hypersurface is: what is the relationship between $\widetilde{b}_{(V(g),f)}(s)$ and $\widetilde{b}_{(V(f),g)}(s)$? We observe that these polynomials are not always equal. Setting $f=x_1^2+x_2^2+x_3^2+x_4^2+x_5^2+x_6^2$ and $g=x_1$, a computer calculation gives
\[
b_{(V(g),f)}(s)=(s+1)\left(s+\frac{5}{2}\right)\quad \textnormal{and}\quad \widetilde{b}_{(V(g),f)}(s)=\left(s+\frac{5}{2}\right).
\]
Comparing this result to Example \ref{eg:sumOfsquares}(1), we see that $\widetilde{b}_{(V(g),f)}(s)$ is not equal to $\widetilde{b}_{(V(f),g)}(s)$, and similarly for the usual $b$-functions.
\end{eg}

\begin{rmk} \label{rmk-cover} If $U' \subseteq Y$ is an open subset with $U' \cap X = U$ nonempty, then we have the divisibility relation
\[ b_{(U,f\vert_U)}(s) \mid b_{(X,f)}(s) \text{ and } \widetilde{b}_{(U,f\vert_U)}(s) \mid \widetilde{b}_{(X,f)}(s). \]

Indeed, we have the commutative diagram
\[ \begin{tikzcd} U' \ar[r, "\Gamma^{\circ}"] \ar[d] & U' \times \A^1_t \ar[d] \\ Y \ar[r, "\Gamma"]  & Y\times \A^1_t\end{tikzcd},\]
so that we have identifications
\[ \Gamma^{\circ}_+ {\rm IC}_{U} \cong (\Gamma_+ {\rm IC}_X)\vert_{U' \times \A^1_t}.\]

Moreover, it is easy to see that
\[ {\rm Gr}_G^j(\Gamma^\circ_+ \cM_U) = ({\rm Gr}_G^j \Gamma_+ \cM)\vert_{U' \times \A^1_t},\]
compatibly with the action of $s$, which proves the desired divisibility relation.

By the sheaf property of ${\rm Gr}_G^0(\cB_f)$, we see moreover that we have equalities
\[ b_{(X,f)}(s)  = {\rm lcm}_{i\in I} b_{(U_i,f\vert_{U_i})}(s),\]
\[ \widetilde{b}_{(X,f)}(s)  = {\rm lcm}_{i\in I} \widetilde{b}_{(U_i,f\vert_{U_i})}(s),\]
where $X = \bigcup_{i\in I} U_i$ is an open cover.
\end{rmk}

\subsection{Microlocal b-functions, vanishing cycles, and characteristic cycles}\label{sec:char} In this subsection, we investigate necessary and sufficient conditions for $\widetilde{b}_{(X,f)}(s)=1$. We write ${\rm Ch}(\cM)\subseteq T^* Y$ for the characteristic variety of a holonomic $\cD$-module $\cM$.

Let $\cN$ be a holonomic $\cD_{Y\times \A^1_t}$-module and let $H = \{t=0\}$.
We say $\cN$ is \emph{noncharacteristic for $t$} (see \cite{htt}*{Section 2.4}) if
\[ {\rm Ch}(\cN) \cap T^*_H(Y\times \A^1_t) \subseteq T^*_{Y\times \A^1_t}(Y\times \A^1_t),\]
i.e. ${\rm Ch}(\cN)$ meets the conormal bundle of $H$ only in the zero section. Given a holonomic $\cD_Y$-module $\cM$ and $f \in \cO_Y$, we say $\cM$ is
\emph{noncharacteristic for $f$} if $\Gamma_+ \cM$ is noncharacteristic
along $t$, where $\Gamma \colon Y \to Y\times \A^1_t$ is the graph embedding.

When $X$ is smooth, the following are equivalent: (1) $\widetilde{b}_{(X,f)}(s)=1$, (2) $X\cap V(f)$ is nonsingular, (3) $\mathcal{O}_X$ is noncharacteristic for $f$ \cite{BriMais}. In the singular ambient setting, we will prove that (3) $\implies$ (1) after replacing $\mathcal{O}_X$ by ${\rm IC}_X$. We will also show that (1) does not necessarily imply (2), but rather $D_{\rm sing} = X_{\rm sing} \cap D$ if $\widetilde{b}_{(X,f)}(s)=1$. We also give an example where the converse of this fact fails, see Example \ref{eg:A1}.

The noncharacteristic property forces total vanishing cycles to be zero.

\begin{prop}\label{p:nonchar} Assume $\cM$ is a holonomic $\cD_Y$-module and $f\colon Y \to \A^1$ is a globally defined function. If $\cM$ is noncharacteristic for $f$, then $\varphi_f(\cM)=0$.

In particular, if ${\rm IC}_X$ is noncharacteristic for $f$, then $\widetilde{b}_{(X,f)}(s)=1$.
\end{prop}

\begin{proof} 
Since $H = \{t=0\}$ is a smooth divisor in the smooth variety $Y\times \A^1_t$ and
$\Gamma_+\cM$ is noncharacteristic along $t$, the module $\Gamma_+\cM$ admits a $V$-filtration along $t$, which is the $t$-adic filtration \cite{KebekusSchnell}*{Lem. 4.17} (though our $V$-filtration is shifted relative to loc. cit.), i.e. 
\[
V^{\alpha}\Gamma_+\cM = t^{\lceil \alpha \rceil-1}\cdot \Gamma_+\cM \quad \textnormal{for $\alpha>1$, and} \quad V^{\alpha}\Gamma_+\cM=\Gamma_+\cM\quad \textnormal{for $\alpha\leq 1$}.
\]
In particular ${\rm Gr}_V^\alpha(\Gamma_+\cM) = 0$ for all $\alpha<1$, so ${\rm Gr}_V^0(\Gamma_+\cM) = 0$. It follows that $\varphi_f(\cM)=0$.
By Lemma \ref{lem:microRootBasics}, setting $\cM={\rm IC}^H_X$ gives $\widetilde{b}_{(X,f)}(s)=1$.
\end{proof}

The following is a generalization of the fact that, when $X$ is smooth, if $\widetilde{b}_{(X,f)}(s) = 1$ then $X\cap V(f)$ is nonsingular \cite{BriMais}.

\begin{prop}\label{propSing} If $\widetilde{b}_{(X,f)}(s) = 1$ then
\[ D_{\rm sing} = X_{\rm sing} \cap D.\]
\end{prop}
\begin{proof} The containment $\supseteq$ is standard. Let $U = X_{\rm reg} \subseteq X$ be the smooth locus of $X$. If $\widetilde{b}_{(X,f)}(s) = 1$, then by the divisibility result as in Remark \ref{rmk-cover}, $\widetilde{b}_{(U,f\vert_U)}(s) = 1$, too. By \cite{BriMais}*{Proposition 2.6}, this implies $f\vert_U$ is smooth. In other words, the singular locus of $D$ is contained in $X\setminus U = X_{\rm sing}$, as desired.
\end{proof}

\begin{eg}
Continuing Example \ref{eg:sumOfsquares}(1), we have $D=V(x_1,x_2^2+x_3^2+x_4^2+x_5^2+x_6^2)\subseteq \A^6$. Here, $X$ is defined by the sum of all six squares and $f=x_1$. We have $\widetilde{b}_{(X,f)}(s)=1$ and $X_{\rm sing} = D_{\rm sing} = \{0\}$, thus illustrating Proposition \ref{propSing}.  
\end{eg}

\begin{eg}\label{eg:A1} The converse of Proposition \ref{propSing} does not hold in general (though it holds if $X$ is smooth). Indeed, if $X=V(xz-y^2)\subseteq \A^3$ and $f=x-z$, then $D=V(x-z,x^2-y^2)$, so $X_{\rm sing} = D_{\rm sing} = \{0\}$. However, a computer calculation shows that $\widetilde{b}_{(X,f)}(s)=(s+1)\neq 1$. We relate this behavior to the characteristic cycle of ${\rm IC}_X$ in Example \ref{eg:CCA1}. 
\end{eg}

\noindent See also Example \ref{eg:sumOfsquares}(3) for an example where $D$ is nonsingular.

In another direction, we relate the property $\widetilde{b}_{(X,f)}(s)=1$ to behavior of the characteristic cycle of ${\rm IC}_X$ in the case of cones. Following \cite{htt}*{Section 2.2}, we write $I({\rm Ch}({\rm IC}_X))$ for the set of irreducible components of ${\rm Ch}({\rm IC}_X)$. The characteristic cycle of ${\rm IC}_X$ is
\begin{equation}
{\rm CC}({\rm IC}_X)=\sum_{C\in I({\rm Ch}({\rm IC}_X))} m_C\cdot [C],   
\end{equation}
where $m_C$ denotes the multiplicity of ${\rm gr}^F({\rm IC}_X)$ along $C$. We say that $[C]$ appears in the characteristic cycle of ${\rm IC}_X$ if $m_C\neq 0$.

Let $Y=\A^n$ and let $X\subseteq Y$ be the cone over a variety in $\P^{n-1}$. Let $L:Y\to \C$ be a linear form and write $W=L^{-1}(0)$. Let $\mathcal{S}$ be a Whitney stratification of $Y$ compatible with $X$, and consider the stratified singular locus of $L$ (see \cite{dimca}*{Definition 4.2.7}):
\[
\operatorname{Sing}_{\mathcal{S}}(L)=\bigcup_{S\in \mathcal{S}}\operatorname{Sing}(L|_S).
\]
We have (see \cite{dimca}*{Proposition 4.2.8}):
\begin{equation}
\operatorname{Supp}(\varphi_L({\rm IC}_X))\subseteq X\cap \operatorname{Sing}_{\mathcal{S}}(L).
\end{equation}

By \cite{braden}*{Theorem 1}, the multiplicity of $[T_{\{0\}}^{\ast}Y]$ in the characteristic cycle of ${\rm IC}_X$ is equal to the multiplicity of $[T^{\ast}_{\{0\}}W]$ in the characteristic cycle of $\varphi_L({\rm IC}_X)$.

In particular, we obtain the following, using Lemma \ref{lem:microRootBasics}.

\begin{prop}\label{prop:Braden}
Suppose that $X$ is the cone over a variety in $\P^{n-1}$ and $L$ is a linear form.

\begin{enumerate}
\item If $\widetilde{b}_{(X,L)}(s)=1$, then $[T_{\{0\}}^{\ast}Y]$ does not appear in $\operatorname{CC}({\rm IC}_X)$.

\item If $\widetilde{b}_{(X,L)}(s)\neq 1$ and $\operatorname{Sing}(L|_S)=\emptyset$ for all $S\neq \{0\}\subseteq \mathbf{A}^n$, then $[T_{\{0\}}^{\ast}Y]$ appears in $\operatorname{CC}({\rm IC}_X)$ with multiplicity at least one.
\end{enumerate}
\end{prop}

\begin{eg}\label{eg:CCA1} Here are a couple examples where we apply Proposition \ref{prop:Braden}.
\begin{enumerate}
    \item Let $X=V(xw-yz)\subseteq \A^4$ and $L=x-w$. A computer calculation shows that $\widetilde{b}_{(X,L)}(s)=1$, recovering the fact that $\operatorname{CC}({\rm IC}_X)$ is irreducible (see \cites{bradGrin, raicu}). 

    \item On the other hand, when $X=V(xz-y^2)\subseteq \A^3$ and $L=x-z$, we saw in Example \ref{eg:A1} that $\widetilde{b}_{(X,L)}(s)=(s+1)\neq 1$. To see that $\operatorname{Sing}(L|_{X_{\operatorname{reg}}})=\emptyset$, we use Lagrange multipliers. A point $p\in X\setminus \{0\}$ is in the stratified singular locus of $L$ iff $dL=\lambda dg$ at $p$ for some $\lambda\in \mathbf{C}$. Hence we want to solve
\[
dx-dz=\lambda(zdx-2ydy+xdz),\quad xz-y^2=0.
\]
It is straightforward to show that this system has no solution. Using Proposition \ref{prop:Braden}, we recover that the characteristic cycle of ${\rm IC}_X$ is reducible (see \cite{raicu}*{Remark 1.5}).
\end{enumerate}
\end{eg}

\section{Purity of local cohomology and Hodge rational homology level}

In this section, we use $\widetilde{b}_{(X,f)}(s)$ to study $\cH^1_f({\rm IC}_X)$ and $\cH^{c+1}_D(\cO_Y)$, where $D=X\cap V(f)$. In particular, we show that the absence of integral roots implies: (1) purity of $\cH^1_f({\rm IC}_X)$ and (2) $\operatorname{HRH}(D)\geq \operatorname{HRH}(X)$. The latter is a consequence of a more refined statement (Proposition \ref{propHRH}), relating $\operatorname{HRH}(D)$ to the Hodge filtration on $\varphi_{f,1}({\rm IC}_X^H)$.

\subsection{Purity of local cohomology}\label{sec:lcXf}

For ease of notation, we write $m=d_Y+c$, so we have equality
\[ {\rm IC}_X^H(-c) = W_{m} \cH^c_X(\cO_Y).\] 
Let $Q=\cH^c_X(\cO_Y)/({\rm IC}_X(-c))$, which by construction has weights $\geq m+1$. By hypothesis on $f$, there is an exact sequence (induced by taking $\cH^\bullet_f(-)$ of the defining short exact sequence for $Q$):
\begin{equation}\label{eqn:sesLC0}
0\to \cH^0_f(Q)\to \cH^1_f({\rm IC}^H_X(-c))\to \cH^1_f(\cH^c_X(\cO_Y))\to \cH^1_f(Q)\to 0.
\end{equation}

Since $\cH^1_f(Q)$ has weights $\geq m+2$ and we have the commutativity \[\cH^0_f \circ W_{m+1}(-) = W_{m+1} \circ \cH^0_f(-),\] 
we obtain a short exact sequence
\begin{equation}\label{eqn:sesLC1}
0\longrightarrow \cH^0_f(\operatorname{Gr}^W_{m+1}(\cH^c_X(\cO_Y)))\longrightarrow W_{m+1}(\cH^1_f({\rm IC}_X^H(-c)))\longrightarrow W_{m+1}(\cH^1_f(\cH^c_X(\cO_Y)))\longrightarrow 0.    
\end{equation}

On the other hand, using the Grothendieck spectral sequence
\begin{equation}\label{e:GSS}
E_2^{i,j}=\cH^i_f(\cH^j_X(\cO_Y))\implies \cH^{i+j}_D(\cO_Y),
\end{equation}
there is an exact sequence
\begin{equation}\label{eqn:sesLC3}
0\to \cH^1_f(\cH^c_X(\cO_Y))\longrightarrow \cH^{c+1}_D(\cO_Y)\longrightarrow \cH^0_f(\cH^{c+1}_X(\cO_Y))\to 0.
\end{equation}

Since $\cH^{c+1}_X(\cO_Y)$ (hence $\cH^0_f(\cH^{c+1}_X(\cO_Y))$) has weights $\geq m+2$ (see \cite{ParkPopa}*{Proposition 6.4}), we get an isomorphism
\[
W_{m+1}(\cH^1_f(\cH^c_X(\cO_Y))) \cong W_{m+1}(\cH^{c+1}_D(\cO_Y)),
\]
and from (\ref{eqn:sesLC1}) there is a short exact sequence
\begin{equation}\label{eqn:sesLC2}
0\longrightarrow \cH^0_f(\operatorname{Gr}^W_{m+1}(\cH^c_X(\cO_Y)))\longrightarrow W_{m+1}(\cH^1_f({\rm IC}_X^H(-c)))\longrightarrow {\rm IC}_D^H(-c-1)\longrightarrow 0,    
\end{equation}
where we used that $W_{m+1}(\cH^{c+1}_D(\cO_Y))={\rm IC}_D^H(-c-1)$.

Combining this information with the results in Section \ref{sec:microBdef}, we obtain the following.

\begin{thm} \label{prop:microAndLocalization}
The module $\cH^1_f({\rm IC}^H_X(-c))$ underlies a pure Hodge module (necessarily of weight $m+1$) if and only if $\varphi_{f,1}({\rm IC}_X) = 0$, if and only if $\widetilde{b}_{(X,f)}(s)$ has no integer roots.

In this situation, $\cH^1_f({\rm IC}^H_X(-c))$ is a direct sum of $\cH^0_f(\operatorname{Gr}^W_{m+1}(\cH^c_X(\cO_Y)))$ and ${\rm IC}_D^H(-c-1)$.
\end{thm}

\begin{proof} We have seen the relationship between $\varphi_{f,1}({\rm IC}_X)$ and integral roots of $\widetilde{b}_{(X,f)}(s)$ in Lemma \ref{lem:microRootBasics}. To study the purity of $\cH^1_f({\rm IC}_X)$, we will use its description as $\cH^1 i^! {\rm IC}_X$, where $i\colon D \to X$ is the closed embedding. We write $M={\rm IC}_X^H(-c)$ and $\mathcal{M}={\rm IC}_X$ for the underlying $\cD_Y$-module.

Since $M$ is a pure Hodge module, we know that the weight filtration on $\psi_{f,1}(M)$ is given by the monodromy filtration along $N = t\de_t$ (using ${\rm Gr}_V^1(\cB_f)$ as a $\cD$-module representative).

Thus, we see that $\varphi_{f,1}(M) = 0$ if and only if $N = 0$ on $\psi_{f,1}(M)$ if and only if $\psi_{f,1}(M)$ is pure. In this case, the isomorphism
\[ \cH^1 i^! M \cong {\rm coker}( \varphi_{f,1}(M) \to \psi_{f,1}(M)(-1)) = \psi_{f,1}(M)(-1),\]
proves the claim.

On the other hand, by \cite{saito90}*{(2.11.10)}, we know that $M_f/M$ is pure if and only if, for all $w >  m +1$ (which is the central weight of the monodromy filtration on the Tate-twisted $\psi_{f,1}(M)(-1)$), we have 
\[ P {\rm Gr}^W_w(\psi_{f,1}(M)(-1)) = 0,\]
where $P {\rm Gr}^W_w (\psi_{f,1}(M)(-1)) = \ker(N^{w-m})$ is the primitive part of the monodromy filtration.

For all $w \geq m+1$, we have the Lefschetz decomposition
\[ {\rm Gr}^W_w(\psi_{f,1}(M)(-1)) = \bigoplus_{\ell \geq 0} N^\ell {\rm Gr}^W_{w+2\ell}(\psi_{f,1}(M)(-1)),\]
and thus we see that ${\rm Gr}^W_w(\psi_{f,1}(M)(-1)) = 0$ for all $w > m+1$. By the isomorphisms
\[ N^\ell \colon {\rm Gr}^W_{m+1+\ell}(\psi_{f,1}(M)(-1)) \to {\rm Gr}^W_{m+1-\ell}(\psi_{f,1}(M)(-1)),\]
for $\ell \geq 0$, we conclude that ${\rm Gr}^W_w (\psi_{f,1}(M)(-1)) = 0$ for all $w \neq m+1$. In particular, $\psi_{f,1}(M)(-1)$ is pure of weight $m+1$, which, as mentioned above, is equivalent to $\varphi_{f,1}(M) = 0$. 

The statement about the direct sum decomposition of $\cH^1_f({\rm IC}^H_X(-c))$ follows from (\ref{eqn:sesLC2}) and the fact that the category of polarizable pure Hodge modules of a given weight is semi-simple.
\end{proof}

Considering the Hodge filtration on (\ref{eqn:sesLC2}), we have an exact sequence for all $p\geq 0$:
\begin{equation}\label{eqn:sesLC4}
0\to \cH^0_f(F_p\operatorname{Gr}^W_{m+1}(\cH^c_X(\cO_Y)))\to F_pW_{m+1}(\cH^1_f({\rm IC}_X^H(-c)))\to F_p{\rm IC}_D^H(-c-1)\to 0.    
\end{equation}

In particular, we have: 

\begin{lem}\label{lem:hodgeEqual}
The module $\cH^0_f(F_p\operatorname{Gr}^W_{m+1}(\cH^c_X(\cO_Y)))$ is zero if and only if
\begin{equation}\label{eqn:Hodgeequal}
F_pW_{m+1}(\cH^1_f({\rm IC}^H_X(-c)))= F_p{\rm IC}_D^H(-c-1).
\end{equation}
\end{lem}

We illustrate Theorem \ref{prop:microAndLocalization} in a few examples.

\begin{eg}\label{eg:detPure}
As in Example \ref{eg:CCA1}(1), let $X=V(xw-yz)\subseteq \A^4$ and let $f=x-w$. Then $\widetilde{b}_{(X,f)}(s)=1$, so by Theorem \ref{prop:microAndLocalization}, we have that $\mathcal{H}^1_f({\rm IC}_X^H)$ is a pure Hodge module. Furthermore, $\operatorname{Gr}^W_{4}(\cH^1_X(\cO_Y))={\rm IC}_{\{0\}}$ by \cite{perlmanraicu}*{Theorem 1.3}, so we have:
\[
\mathcal{H}^1_f({\rm IC}_X^H)\cong {\rm IC}^H_D(-1)\oplus {\rm IC}^H_{\{0\}}(-2).
\]
Here, ${\rm IC}_D$ is the direct image along $V(f)\hookrightarrow \A^4$ of the $\cD$-module ${\rm IC}_{V(x^2-yz)}$. 
\end{eg}

\begin{eg}
As in Example \ref{eg:CCA1}(2), setting $X=V(xz-y^2)\subseteq \A^3$ and $f=x-z$, we have that $\widetilde{b}_{(X,f)}(s)=(s+1)$. Hence, $\mathcal{H}^1_f({\rm IC}_X^H)$ is not a pure Hodge module. 
\end{eg}

\begin{eg}\label{eg:E8}
Let $X=V(x^2+y^3+z^5)\subseteq \A^3$ be the $\mathsf{E}_8$ surface singularity and let $f=z$. A computer calculation shows that
\[
\widetilde{b}_{(X,f)}(s)=\left(s+\frac{1}{6}\right)\left(s+\frac{11}{6}\right).
\]
By Theorem \ref{prop:microAndLocalization}, we have that $\mathcal{H}^1_f({\rm IC}_X^H)$ is a pure Hodge module. As $X$ is a finite quotient singularity, it is a rational homology manifold, so $\mathcal{H}^1_X(\cO_Y)={\rm IC}_X^H(-1)$. It follows from the spectral sequence (\ref{e:GSS}) that 
\[
\mathcal{H}^2_D(\cO_Y)\cong \mathcal{H}^1_f({\rm IC}_X^H(-1))\cong {\rm IC}_D^H(-2),\quad \textnormal{and}\quad \mathcal{H}^i_D(\cO_Y)=0\;\textnormal{for $i>2$.} 
\]
As $D=V(z,x^2+y^3)$, we recover that this cuspidal singularity is a rational homology manifold. Corollary \ref{cor-HRHDivisor} below generalizes this phenomenon.
\end{eg}

\subsection{The Hodge rational homology level}
We recall the Hodge rational homology level of \cite{DOR1} (see also Section \ref{sec:LCpre}). If $Z\subseteq Y$ is a closed equidimensional subvariety of dimension $d_Z$, then ${\rm HRH}(Z) \geq k$ if and only if the natural composition of morphisms
\begin{equation}\label{eqn:HRHprop}
{\rm IC}_Z^H \to \cH^0 \mathbf D_Z(\Q_Z^H[d_Z])(-d_Z) \to \mathbf D_Z(\Q_Z^H[d_Z])(-d_Z),
\end{equation}
gives a quasi-isomorphism after applying $F_{k+d_Y-d_Z}(-)$. The property $\operatorname{HRH}(Z)\geq k$ is referred to as the condition $D_k$ in \cite{ParkPopa}, where it is shown that it implies certain symmetries of the Hodge--Du Bois diamond of $Z$. 

The paper \cite{CDOCCI} introduced a related invariant $c(Z) \in \Z_{\geq -1} \cup \{\infty\}$. By definition, $c(Z) \geq k$ if and only if the natural morphism
\[ \cH^0 \mathbf D_Z(\Q_Z^H[d_Z])(-d_Z) \to \mathbf D_Z(\Q_Z^H[d_Z])(-d_Z),\]
gives a quasi-isomorphism after applying $F_{k+d_Y-d_Z}(-)$. In particular, we have
\[ c(X) \geq {\rm HRH}(X).\]
Moreover, if $Z' \subseteq Z$ is a Cartier divisor in $Z$, then \cite{DirksMultiplier}*{Lem. 3.4} shows that 
\[ c(Z')\geq c(Z).\]

Let $D\subseteq X\subseteq Y$ be as in our basic set-up. We now show that ${\rm HRH}(D)\geq {\rm HRH}(X)$ when $\widetilde{b}_{(X,f)}(s)$ has no integer roots.

Suppose that ${\rm HRH}(X) \geq k$, and let $i \colon D = \{f=0\} \to X$ be the closed embedding defined by $f$. Then applying $i^!$ to (\ref{eqn:HRHprop}) gives an isomorphism
\begin{equation} \label{eqn:HRHprop1} F_{k+c} i^! {\rm IC}_X^H \cong F_{k+c} i^!\mathbf D_X(\Q_X^H[\dX])(-\dX) = F_{k+c+1} \mathbf D_D(\Q_D^H[\dD])(-\dD)[-1],\end{equation}
where the shift by $1$ on the rightmost object is due to the difference $\dX - \dD = 1$, arising from the Tate twists.

Since $f$ restricts to a nonzero noninvertible function on $X$, it follows that ${\rm IC}_X$ has no submodule supported on $D$, and hence $i^! {\rm IC}_X^H$ is concentrated in cohomological degree $1$, isomorphic to $\mathcal{H}^1_f({\rm IC}_X)$. Thus, we can rewrite the isomorphism (\ref{eqn:HRHprop1}) as
\begin{equation} \label{eqn:HRHprop2} 
F_{k+c} \mathcal{H}^1_f({\rm IC}_X) \cong  F_{k+c+1} \mathbf D_D(\Q_D^H[\dD])(-\dD).
\end{equation}

We have the commutative diagram
\[ \begin{tikzcd} W_{\dX+1} \mathcal{H}^1_f({\rm IC}_X) \ar[r] \ar[d] & {\rm IC}_D^H(-1) \ar[d] \\ \mathcal{H}^1_f({\rm IC}_X) \ar[r] & \cH^0 \mathbf D_D(\Q_D^H[\dD])(-\dD)(-1)\end{tikzcd},\]
where the top row is $W_{\dX+1}(-)$ applied to the bottom row. We have explained above that the bottom morphism is an isomorphism at $F_{k+c}(-)$, which implies that the top morphism is also an isomorphism at $F_{k+c}(-)$.

Thus, we conclude that ${\rm HRH}(D) \geq k$ if and only if the natural inclusion
\[ W_{\dX+1} \mathcal{H}^1_f({\rm IC}_X) \hookrightarrow \mathcal{H}^1_f({\rm IC}_X),\]
is an isomorphism at $F_{k+c}$. Indeed, note that $c(D) \geq c(X) \geq {\rm HRH}(X) \geq k$ as $D$ is Cartier inside $X$, so that in order to understand whether ${\rm HRH}(D) \geq k$ it suffices to compare $\cH^0(\mathbf D_D(\Q_D^H[\dD]))(-\dD)$ to its lowest weight piece ${\rm IC}_D^H$.

The following is a generalization of \cite{DOR2}*{Theorem A} to the singular ambient setting.

\begin{prop}\label{propHRH} The natural inclusion
\[ W_{\dX+1} \mathcal{H}^1_f({\rm IC}_X) \hookrightarrow \mathcal{H}^1_f({\rm IC}_X)\]
is an isomorphism at $F_{k+c}(-)$ if and only if 
\[ F_{k+c} \varphi_{f,1}({\rm IC}_X^H) = F_{k+c+1} {\rm Gr}_V^0(\cB_f) =  0.\]
\end{prop}
\begin{proof} The Hodge and weight filtrations on $\mathcal{H}^1_f({\rm IC}_X)$ are determined by the morphism
\[ {\rm Gr}_V^0(\cB_f) \xrightarrow[]{t} {\rm Gr}_V^1(\cB_f),\]
where since $\cB_f$ is pure, the weight filtration is simply the monodromy filtration. The proof is identical at this point to \cite{DOR2}*{Pf. of Thm. A}.
\end{proof}

\begin{cor} \label{cor-HRHDivisor} Assume ${\rm HRH}(X) \geq k$. Then 
\[ {\rm HRH}(D) \geq k \iff p(\varphi_{f,1}({\rm IC}_X^H)) \geq k-d_X +2,\]
where 
\[
p(\cM) = \min\{k \mid F_k \cM \neq 0\}- d_Y = \min\{k \mid {\rm Gr}^F_k {\rm DR}(\cM) \neq 0\},
\]
for any filtered left $\cD_Y$-module  $(\cM,F)$.

In particular, if $\widetilde{b}_{(X,f)}(s)$ has no integer roots, then ${\rm HRH}(D) \geq {\rm HRH}(X)$.    
\end{cor}

When $X$ is smooth, the last statement of Corollary \ref{cor-HRHDivisor} recovers the fact that $D$ is a rational homology manifold if $\widetilde{b}_f(s)$ has no integer roots.

Moreover, Corollary \ref{cor-HRHDivisor} gives a slight refinement and alternative proof of \cite{ParkGIT}*{Prop. 11.2}.

\begin{cor} Assume $\min\{{\rm HRH}(X),{\rm HRH}(D)\} \geq m-1$. Then 
\[ {\rm Gr}^F_{p} {\rm DR}(\varphi_{f,1}({\rm IC}_X^H)) = 0 \text{ for all } p \geq -m.\]
If ${\rm HRH}(X) \geq m$ and ${\rm HRH}(D) \geq m-1$, then 
\[ {\rm Gr}^F_{p} {\rm DR}(\varphi_{f,1}(\Q_X^H[d_X])) = 0 \text{ for all } p \geq -m.\]
\end{cor}
\begin{proof} Assume $\min\{{\rm HRH}(X), {\rm HRH}(D)\} \geq m-1$. Then Corollary \ref{cor-HRHDivisor} gives an inequality 
\[ p(\varphi_{f,1}({\rm IC}_X^H)) \geq (m-1)-d_X + 2,\]
or in other words,
\[ {\rm Gr}^F_p {\rm DR}( \varphi_{f,1}({\rm IC}_X^H)) = 0 \text{ for all } p \leq m - d_X .\]
Dually, using that $\varphi_{f,1}(-)$ commutes with duality and ${\rm IC}_X^H$ is polarizable of weight $d_X$, we get
\[ {\rm Gr}^F_{-p} {\rm DR}( \varphi_{f,1}({\rm IC}_X^H(d_X))) = {\rm Gr}^F_{-(p+d_X)} {\rm DR}(\varphi_{f,1}({\rm IC}_X^H)) = 0 \text{ for all } p \leq m - d_X ,\]
or, by reindexing,
\[ {\rm Gr}^F_{p} {\rm DR}(\varphi_{f,1}({\rm IC}_X^H)) = 0 \text{ for all } p \geq -m.\]
If we assume moreover that ${\rm HRH}(X) \geq m$, then we know that the morphism
\[ {\rm IC}_X^H(-c) = W_{d_Y +c} \cH^c_X(\cO_Y) \to \cH^c_X(\cO_Y) \to \mathbf{R}\Gamma_X(\cO_Y)[c],\]
induces an isomorphism after applying $F_{m}(-)$. Hence, if $\Gamma \colon Y \to Y \times \A^1_t$ is the graph embedding along $f$, the induced morphism
\[ \cB_f(-c) \to \Gamma_+\cH^c_X(\cO_Y) \to \Gamma_+ \mathbf{R}\Gamma_X(\cO_Y)[c],\]
is an isomorphism after applying $F_{m+1}(-)$. So the same is true after applying ${\rm Gr}_V^0(-)$. Recalling that there is a shift in the Hodge filtration for unipotent vanishing cycles (when using left filtered $\cD$-modules), we conclude that we have an isomorphism
\[ F_m \varphi_{f,1}( {\rm IC}_X^H(-c)) \cong F_m \varphi_{f,1}(\mathbf{R}\Gamma_X(\cO_Y))[c],\]
which translates to quasi-isomorphisms
\[ {\rm Gr}^F_{p-d_Y} {\rm DR}(\varphi_{f,1}( {\rm IC}_X^H(-c))) \cong {\rm Gr}^F_{p-d_Y} {\rm DR}( \varphi_{f,1}(\mathbf{R}\Gamma_X(\cO_Y))[c]) \text{ for all } p \leq m.\]
If we recall that $\mathbf{R}\Gamma_X(\cO_Y) = i_* i^! \Q_Y^H[d_Y] = i_* \mathbf D(\Q_X^H[d_X])(-d_Y)[-c]$, then it is more natural to Tate twist, yielding the quasi-isomorphisms
\[ {\rm Gr}^F_{p} {\rm DR}(\varphi_{f,1}( {\rm IC}_X^H(d_X))) \cong {\rm Gr}^F_{p} {\rm DR}( \varphi_{f,1}(\mathbf D(\Q_X^H[d_X]))) \text{ for all } p \leq m.\]
Finally, by applying duality (and once more using that ${\rm IC}_X^H$ is polarizable of weight $d_X$), we get isomorphisms
\[ {\rm Gr}^F_{-p} {\rm DR}( \varphi_{f,1}({\rm IC}_X^H)) \cong {\rm Gr}^F_{-p} {\rm DR}( \varphi_{f,1}(\Q_X^H[d_X])) \text{ for all } p\leq m,\]
and so by re-indexing and applying the first statement of the corollary, we conclude the desired vanishing.
\end{proof}

\begin{eg}\label{eg:HRHbound}
We observe Corollary \ref{cor-HRHDivisor} in a couple examples.
\begin{enumerate}
\item Continuing Example \ref{eg:detPure}, where $X=V(xw-yz)\subseteq \A^4$, $f=x-w$, and $\widetilde{b}_{(X,f)}(s)=1$, we have $\operatorname{HRH}(X)=0$, so Corollary \ref{cor-HRHDivisor} gives that ${\rm HRH}(D)\geq 0$ (which, since $\dim D = 2$, is equivalent to $D$ being a rational homology manifold). As $D=V(x-w,x^2-yz)$, this is a shadow of the fact that the $\mathsf{A}_1$ surface singularity is a rational homology manifold. 

\item Let $X=V(x_1^2+x_2^2+x_3^2+x_4^2+x_5^2)\subseteq \A^5$ and let $f=x_5$. It is well known that $X$ is a rational homology manifold, while $D$ is not. This is reflected in the microlocal $b$-function. A computer calculation shows that
\[
 \widetilde{b}_{(X,f)}(s)=(s+3),
\]
which has an integral root. Hence, the vanishing cycles are nonzero.
\end{enumerate}
\end{eg}

For another example, see Example \ref{eg:E8}.

\section{The minimal exponent, Hodge filtration, and rational singularities}\label{sec:minExp}

In this section, we define the minimal exponent $\widetilde{\alpha}(X,f)$ and use it to study the Hodge and weight filtrations on $\mathcal{H}^1_f({\rm IC}_X^H)$. As an application, we characterize rational singularities of $D$, assuming that $X$ has rational singularities. In another direction, we prove a Thom--Sebastiani result for minimal exponents.

We write $M={\rm IC}^H_X$, and write $\cM ={\rm IC}_X$ for the underlying $\cD_Y$-module. Hence, the Hodge filtration satisfies $F_{c-1}(\cM)=0$ and $F_c(\cM)\neq 0$.

\subsection{Comparing filtrations on local cohomology} 

If $X$ is smooth, the theory of Hodge ideals \cite{HodgeIdeals} (comparing $F_k \cO_X(*D)$ to $P_k \cO_X(*D) = \cO_X((k+1)D)$) is equivalent to the comparison of the Hodge filtration on local cohomology $F_k \cH^1_f(\cO_X)$ to the pole order filtration $\sum_{i=0}^k \frac{\cO_X}{f^{k-i+1}} \subseteq \frac{\cO_X(*D)}{\cO_X} = \cH^1_f(\cO_X)$. By \cite{SaitoHodgeIdeal}, triviality of Hodge ideals can be understood via the microlocal $V$-filtration. In this subsection, we generalize this relationship.

We define the naive Hodge filtration $O_{\bullet}$ on $\mathcal{H}^1_f(\cM)$ as follows:
\begin{equation}
O_k \cH^1_f(\cM) = \sum_{i=0}^k \frac{F_{i} \cM}{f^{k-i+1}}.    
\end{equation}
Thus, $O_{c-1}=0$ and $O_c\neq 0$. Note that $f^{k+1} O_{c+k} = 0$, however, the converse is far from true. In fact, we cannot compare $F_k$ on $\cH^1_f(\cM)$ with the usual pole order filtration, as the latter is not $\cO_Y$-coherent. Note also that $f O_k = O_{k-1}$, in analogy with what is known if $X$ is smooth. 

For the following statement, define $(\cM,F)$ to be $f$-saturated up to level $k$ if
\begin{equation}
 f\eta \in F_j\cM \implies \eta\in F_j\cM\quad \textnormal{for all $j\leq k$.}   
\end{equation}
Equivalently, multiplication by $f$ is injective on $\cM/F_j\cM$ for all $j\leq k$. If ${\rm Gr}^F_\bullet \cM$ has no $f$-torsion (for example, if ${\rm Ch}(\cM)$ has no component of the form $T^*_Z Y$ with $Z\subseteq V(f)$), then $(\cM,F)$ is $f$-saturated up to level $k$, for all $k\in \Z$. 

We note that the hypotheses on $X$ and $f$ in our basic set-up imply that $\cM$ is $f$-torsion-free.

\begin{thm}\label{thm:naiveOrd} \begin{enumerate} \item We have containment $F_k \cH^1_f(\cM) \subseteq O_k \cH^1_f(\cM)$ for all $k\in \Z$.

\item \label{itm-FVO} If $F_{k+1}\cB_f \subseteq V^1\cB_f$, then $F_k \cH^1_f(\cM) = O_k \cH^1_f(\cM)$.

\item \label{itm-FVWO} If $(s+1)(F_{k+1}\cB_f) \subseteq V^{>1}\cB_f$, then $F_k W_{d_X+1} \cH^1_f(\cM) = O_k \cH^1_f(\cM)$.
\end{enumerate} 
\noindent The converses to \eqref{itm-FVO} and \eqref{itm-FVWO} hold if $(\cM,F)$ is $f$-saturated up to level $k$.
\end{thm}
\begin{proof} Recall that $\cH^1_f(\cM,F)$ can be computed via the (strict injective) morphism
\[ (V^0 \cB_f, F[-1]) \xrightarrow[]{t} (V^1 \cB_f,F[-1])\]
where the shift on the Hodge filtration is due to our use of left $\cD$-modules.

At the $\cD$-module level, we have the isomorphism
\begin{equation} \label{eq-DModIso} {\rm coker}(\cB_f \xrightarrow[]{t}\cB_f) \cong \cH^1_f(\cM)\end{equation}
given by sending $\sum_{i=0}^k m_i \de_t^i \delta_f \in \cB_f$ to the element
\[ \sum_{i=0}^k \frac{(-1)^i m_i}{f^{i+1}} \in \cH^1_f(\cM).\]

To prove (1) and (2), observe that $O_k \cH^1_f(\cM)$ is the image of $F_{k+1}\cB_f$ under the isomorphism \eqref{eq-DModIso}, whereas the description in terms of the $V$-filtration shows that $F_k \cH^1_f(\cM)$ is the image of $F_{k+1}V^1 \cB_f$. Thus, the containment $F_{k+1}V^1\cB_f \subseteq F_{k+1} \cB_f$ gives the containment $F_k \cH^1_f(\cM) \subseteq O_k \cH^1_f(\cM)$ for all $k \in \Z$. Moreover, we see that if $F_{k+1} \cB_f \subseteq V^1 \cB_f$, then $F_k\cH^1_f(\cM) = O_k \cH^1_f(\cM)$. 

Next, we prove (3). To understand the weight filtration on $\cH^1_f(\cM)$, we use the representation 
\[ t \colon {\rm Gr}_V^0(\cB_f,F[-1]) \to {\rm Gr}_V^1(\cB_f,F[-1]), \]
which is bi-strict with respect to the Hodge and weight filtrations. As $\cB_f$ is pure, the weight filtration on ${\rm Gr}_V^\lambda(\cB_f)$ is the monodromy weight filtration associated to $(s+\lambda)$ (suitably shifted). By following the argument \cite{CDO}*{Pf. of Thm B}, we see that $F_k W_{d_X+1} \cH^1_f(\cM)$ is the image of those elements $m \in F_{k+1} V^1 \cB_f$ such that $(s+1)m \subseteq V^{>1} \cB_f$. 

Assume now that $(s+1)(F_{k+1} \cB_f) \subseteq V^{>1} \cB_f$. As $(s+\lambda)$ is nilpotent on ${\rm Gr}_V^\lambda(\cB_f)$ for all $\lambda \in \Q$, this implies that $F_{k+1} \cB_f \subseteq V^1 \cB_f$. Hence, we see that $F_{k+1} \cB_f$ maps into $F_k W_{d_X+1} \cH^1_f(\cM)$, proving the equality with $O_k \cH^1_f(\cM)$.

For the converse to \eqref{itm-FVO}, we use induction on $k$, the base case being $k=c-1$. Assume inductively that $F_k \cB_f \subseteq V^1 \cB_f$. This implies, by strict surjectivity of
\[ \de_t\colon ({\rm Gr}_V^{\chi+1}(\cB_f),F[1]) \to ({\rm Gr}_V^{\chi}(\cB_f),F),\]
for $\chi < 0$, that $F_{k+1} \cB_f \subseteq V^0 \cB_f$. 

Thus, we have an isomorphism
\begin{equation}\label{e:presOF}
{\rm coker}(F_{k+1} \cB_f \xrightarrow[]{t} F_{k+1}V^1\cB_f) \cong F_k \cH^1_f(\cM) = O_k \cH^1_f(\cM).    
\end{equation} 

Let $m \in F_{k-i} \cM$, so that our goal is to prove $m\de_t^i\delta_f \in V^1\cB_f$. Of course, $i \leq k$, otherwise the element is $0$. By definition of $O_k\cH^1_f(\cM)$ and the assumption that $O_k\cH^1_f(\cM)= F_k\cH^1_f(\cM)$, we have $\frac{m}{f^{i+1}} \in F_k\cH^1_f(\cM)$. By (\ref{e:presOF}), there exists some $\sum_{j=0}^k m_j \de_t^j \delta_f \in F_{k+1} V^1\cB_f$ and some $\eta = \sum_{j=0}^b \eta_j \de_t^j \delta_f \in \cB_f$ such that
\begin{equation} \label{eq-1} m \de_t^i \delta_f = \sum_{j=0}^k m_j \de_t^j \delta_f + t(\eta).\end{equation}
If we expand this equality and compare the coefficient of $\de_t^\ell \delta_f$ on each side, we get
\[ m_\ell + (f\eta_\ell - (\ell+1)\eta_{\ell+1}) = \begin{cases} 0 & \ell \neq i \\ m & \ell = i \end{cases}.\]

\noindent \textit{Claim 1:} We have $\eta_j = 0$ for all $j > k$.

To see this, if $j > k$, then we get
\[ 0 = f\eta_j - (j+1) \eta_{j+1},\]
and so by descending induction on $j$, we see $f\eta_j=0$, hence, as $\cM$ is $f$-torsion free, we get the desired vanishing.

\noindent \textit{Claim 2:} We have $\eta_j \in F_{k-j}\cM$.

We prove Claim 2 by descending induction, the case $j>k$ following from Claim 1. Now, we have equality
\[ f\eta_\ell  = \begin{cases} -m_\ell+ (\ell+1)\eta_{\ell+1} & \ell \neq i, \\ -m_\ell +(\ell+1)\eta_{\ell+1}+m & \ell = i, \end{cases}\]
and in any case, the right-hand side lies in $F_{k-\ell}\cM$. Since $k-\ell\leq k$, the assumed $f$-saturation of
$F_{k-\ell}\cM\subseteq\cM$ implies that
$\eta_\ell\in F_{k-\ell}\cM$, proving Claim 2.

By definition of the Hodge filtration on $\cB_f$, this means $\eta \in F_{k+1}\cB_f = F_{k+1}V^0 \cB_f$. Thus, 
\[
m\de_t^i\delta_f = \sum_{j=0}^k m_j \de_t^j \delta_f + t\eta \in F_{k+1} V^1\cB_f + t(F_{k+1} V^0 \cB_f) = F_{k+1} V^1\cB_f, 
\]
completing the proof of the converse of (2).

Finally, we prove the converse of \eqref{itm-FVWO}. Assume that $F_k W_{d_X+1} \cH^1_f(\cM) = O_k \cH^1_f(\cM)$. By the converse of (2), this implies that $F_{k+1} \cB_f \subseteq V^1 \cB_f$. Let $m \in F_{k-i} \cM$ and consider $m \de_t^i \delta_f$ as above. We can write
\[ m \de_t^i \delta_f = \sum_{j=0}^k m_j \de_t^j \delta_f + t \eta,\]
where $\eta \in F_{k+1} V^0 \cB_f$ and $(s+1)(\sum_{j=0}^k m_j \de_t^j \delta_f) \in V^{>1} \cB_f$. As $F_{k+1} \cB_f \subseteq V^1 \cB_f$, we see that $\eta \in F_{k+1} V^1 \cB_f$. Hence, by applying $(s+1)$, we get
\[ (s+1)(m\de_t^i\delta_f) = (s+1)(\sum_{j=0}^k m_j \de_t^j\delta_f) + t (s+1)\eta \in V^{>1} + t V^1 \subseteq V^{>1},\]
proving the converse of \eqref{itm-FVWO}.
\end{proof}

\subsection{Minimal exponent, local cohomology, and rational singularities} Using notation as in our basic set-up, we define the minimal exponent of the pair $(X,f)$:
\begin{equation}
\widetilde{\alpha}(X,f) = \min \{ \lambda \mid \widetilde{b}_{(X,f)}(-\lambda) = 0\}.    
\end{equation}

Recall the notion of \emph{log canonical threshold} for a pair $(X,f)$ (where $X$ is not necessarily $\Q$-Gorenstein), defined as the maximal $\lambda$ such that the multiplier module $\cJ(\omega_X,f^{\lambda-\varepsilon})$ is equal to $\omega_X^{\rm GR}$. By Proposition \ref{prop-bFunctionDivide} and \cite{DirksMultiplier}*{Thm. 1.8(3)}, we see easily that we have the equality 
\begin{equation} \label{eq-lctMinMinExp1} {\rm lct}(\omega_X,f) = \min\{\widetilde{\alpha}(X,f),1\}.\end{equation}
We record a lemma relating the minimal exponent to the $V$-filtration on $\cB_f$.

\begin{lem}\label{l:minExp}
Using notation as above, we have
\begin{enumerate}
\item $\widetilde{\alpha}(X,f)\geq 1$ if and only if $F_{c+1}\cB_f \subseteq V^1\cB_f$,   

\item $\widetilde{\alpha}(X,f)> 1$ if and only if
\[
F_{c+1}\cB_f \subseteq V^1\cB_f,\quad \textnormal{and}\quad (s+1)(F_{c+1}\cB_f)\subseteq V^{>1}\cB_f.
\]
\end{enumerate}    
\end{lem}

\begin{proof}
We have $F_{c+1}\cB_f = F_c\cM\,\delta_f$ since $F_{c-1}\cM = 0$.   

(1) By Lemma \ref{lem:microRootBasics}(1), $\widetilde\alpha(X,f) \geq 1$ if and only if
${\rm Gr}_V^\lambda {\rm Gr}_G^0(\tcB_f) = 0$ for all $\lambda < 1$. Using
Lemma \ref{lem:microGrBasics}(3), this is equivalent to
${\rm Gr}_V^\lambda {\rm Gr}_G^0(\cB_f) = 0$ for all $\lambda < 1$. 

Note that $F_c \cM \delta_f \subseteq V^{>0}\cB_f$ as ${\rm IC}_X^H$ has strict support not contained in $D$. Thus, $G^1 \cB_f = V^1 \cR \cdot (F_c \cM \delta_f) \subseteq V^{>1}\cB_f$, and so for any $\lambda \leq 1$, we have
\[ {\rm Gr}_V^\lambda {\rm Gr}_G^0(\cB_f) = \frac{V^\lambda G^0 \cB_f}{V^{>\lambda}G^0\cB_f + V^\lambda G^1 \cB_f} = \frac{V^\lambda G^0 \cB_f}{V^{>\lambda}G^0 \cB_f},\]
using that $G^1 V^\lambda \cB_f \subseteq G^0 \cB_f \cap V^{>1}\cB_f \subseteq G^0V^{>\lambda}\cB_f$. Now, we see that the vanishing ${\rm Gr}_V^\lambda {\rm Gr}_G^0(\cB_f)$ for all $\lambda < 1$ is equivalent to equality $V^{\lambda} G^0 \cB_f = V^{>\lambda}G^0 \cB_f$ for all $\lambda < 1$, which, by exhaustiveness of the $V$-filtration, is equivalent to $G^0\cB_f \subseteq V^1 \cB_f$, as claimed.

(2) We have $\widetilde\alpha(X,f) > 1$ if and only if, in addition, $-1$ is not a root of
$\widetilde b_{(X,f)}(s)$, which by Lemma \ref{lem:microRootBasics}(1) means
${\rm Gr}_V^1 {\rm Gr}_G^0(\tcB_f) = 0$. As $\widetilde{\alpha}(X,f) \geq 1$ implies $G^0 \tcB_f \subseteq V^1\tcB_f$, this vanishing is equivalent to the containment $G^0 \tcB_f \subseteq V^{>1}\tcB_f$, or equivalently, $F_c \cM \delta_f \subseteq V^{>1}\tcB_f$. Applying $\de_t$, this is equivalent to the containment $(F_c \cM)\de_t \delta_f \subseteq V^{>0}\tcB_f$, but it is always contained in $\cB_f$, so this is moreover equivalent to $(F_c \cM)\de_t \delta_f \subseteq V^{>0}\cB_f$.

On the other hand, $(s+1)F_{c+1}\cB_f = t (F_c \cM)\de_t\delta_f \subseteq V^{>1}\cB_f = t V^{>0}\cB_f$ is equivalent (using $t$-torsion freeness) to the containment 
\[ (F_c \cM)\de_t\delta_f \subseteq V^{>0}\cB_f,\]
as claimed.
\end{proof}

The Grauert--Riemenschneider sheaf of $X$ is defined as (see \cite{saitoComplexes}*{Page 3})
\begin{equation}
\omega_X^{\operatorname{GR}}=F_c({\rm IC}_X)\otimes_{\cO_Y}\omega_Y.    
\end{equation}
As such, we define
\begin{equation}
\cO_X^{\operatorname{GR}}=F_c({\rm IC}_X).    
\end{equation}
Writing $D=X\cap V(f)$ as in our basic set-up, we have by Theorem \ref{thm:naiveOrd} that
\begin{equation}\label{eqn:ordContain}
F_{c}(\cH^1_f(\cM))\subseteq \cO^{\operatorname{GR}}_X(D)/\cO_X^{\operatorname{GR}}.
\end{equation}
Combining Lemma \ref{l:minExp} and Theorem \ref{thm:naiveOrd}, we get that the minimal exponent may be used to compare these two modules:

\begin{thm}\label{thm:minExpLC} If $\widetilde{\alpha}(X,f) \geq 1$, then
\begin{equation}\label{eqn:pairGR1}
F_{c}(\cH^1_f(\cM))=\cO^{\operatorname{GR}}_X(D)/\cO_X^{\operatorname{GR}}.
\end{equation}
If $\widetilde{\alpha}(X,f)>1$, then
\begin{equation}\label{eqn:pairGR2}
F_{c}(W_{d_X+1}\cH^1_f(\cM))=\cO^{\operatorname{GR}}_X(D)/\cO_X^{\operatorname{GR}}.
\end{equation}
The converse to either statement holds if multiplication by $f$ is injective on ${\rm IC}_X/\cO_X^{\operatorname{GR}}$.
\end{thm}

\begin{rmk} \label{rmk-HigherPossibilities} When $X$ is smooth, there are related statements with $\widetilde{\alpha}(X,f) \geq p+1$ \cite{MPhigher}. Such conclusions are unavailable with the current definition of $\widetilde{\alpha}(X,f)$. Indeed, this minimal exponent, by definition, controls which piece of the microlocal $V$-filtration which contains the lowest Hodge piece of $\cB_f$, and unless the Hodge filtration of ${\rm IC}_X^H$ is generated by this lowest Hodge piece, we cannot hope to obtain any information about the higher Hodge pieces of $\cH^1_f(\cM)$.

A natural remedy for this issue is to consider alternative microlocal Bernstein--Sato polynomials $\widetilde{b}_{(X,f,p)}(s)$ for $p\geq 0$, defined not by the lowest Hodge piece $F_c{\rm IC}_X^H$, but by a functional equation involving a higher piece $F_{c+p}{\rm IC}_X^H$. These polynomials may be dependent on the embedding of $X$ into a smooth variety $Y$, though that does not mean they have no utility.
\end{rmk}

Applying Lemma \ref{lem:hodgeEqual} we obtain:

\begin{cor}\label{cor:GRadjunt}
Suppose that $\cH^0_f(F_0(\operatorname{Gr}^W_{m+1}(\cH^c_X(\cO_Y))))=0$. If $\widetilde{\alpha}(X,f)>1$ then 
\[
\cO^{\operatorname{GR}}_D\cong \cO^{\operatorname{GR}}_X(D)/\cO_X^{\operatorname{GR}}.
\]
The converse holds if multiplication by $f$ is injective on ${\rm IC}_X/\cO_X^{\operatorname{GR}}$.
\end{cor}

A sufficient condition for $\cH^0_f(F_0(\operatorname{Gr}^W_{m+1}(\cH^c_X(\cO_Y))))=0$ is ${\rm HRH}(X)\geq 0$. For instance, if $X$ has rational singularities, this hypothesis is met.

Next, we turn to a characterization of rational singularities of $D$, in the situation that the ambient $X$ has rational singularities.

 \begin{thm}\label{thm:ratSing}
Suppose that $X$ has rational singularities. Then $D$ has rational singularities if and only if $\widetilde{\alpha}(X,f)>1$.    
\end{thm}

\begin{proof}
As $X$ has rational singularities, we know that $\omega_X^{\operatorname{GR}}=\omega_X$, $D$ is Cohen--Macaulay and the adjunction formula holds:
\[ \omega_D = \omega_X(D)/\omega_X.\]

Moreover, we have the vanishing $F_0(\operatorname{Gr}^W_{m+1}(\cH^c_X(\cO_Y)))=0$ by \cite{DOR1}*{Corollary E}.

Finally, multiplication by $f$ is injective on ${\rm IC}_X/\cO_X^{\operatorname{GR}}$. In fact, it is injective on the module $\cH^c_X(\cO_Y)/ F_0 \cH^c_X(\cO_Y)$ by  \cite{CDOHigher}, and we have the inclusion
\[ {\rm IC}_X/\cO_X^{\operatorname{GR}} \hookrightarrow \cH^c_X(\cO_Y)/ F_0 \cH^c_X(\cO_Y)\]
by strictness of the Hodge filtration.

Thus, by Corollary \ref{cor:GRadjunt}, we see that $\widetilde{\alpha}(X,f)>1$ if and only if $\omega_D^{\operatorname{GR}}=\omega_X(D)/\omega_X$, which by the adjunction formula is equivalent to the equality $\omega_D^{\operatorname{GR}}=\omega_D$. As $D$ is Cohen-Macaulay, the latter is equivalent to $D$ having rational singularities, which concludes the proof.
\end{proof}

\begin{rmk} If $X$ has rational singularities, then both conditions in the statement of Theorem \ref{thm:ratSing} imply that $D$ has Du Bois singularities. To see this, recall by \eqref{eq-lctMinMinExp1} that $\widetilde{\alpha}(X,f) > 1$ implies ${\rm lct}(\omega_X,f) =1$, so one can apply \cite{DirksMultiplier}*{Cor. 1.5} to conclude that $D$ has Du Bois singularities. Using this, we can give another proof of Theorem \ref{thm:ratSing}: it suffices to prove that if $D$ is Du Bois and $X$ is rational, then ${\rm HRH}(D) \geq 0$ if and only if $\widetilde{\alpha}(X,f) > 1$. This then follows from Corollary \ref{cor-HRHDivisor}.
\end{rmk}

\begin{eg}\label{eg:ratSing}
We observe Theorem \ref{thm:ratSing} in several examples.

\begin{enumerate}
\item Let $X=V(y^2-xz)\subseteq \A^3$ and let $f=y$. Then $D=V(y,xz)$, which is a normal crossing singularity at the origin. The microlocal $b$-function is $\widetilde{b}_{(X,f)}(s)=(s+1)$, so $\widetilde{\alpha}(X,f)=1$, so $D$ does not have rational singularities. 

\item For all pairs $(X,f)$ in Example \ref{eg:sumOfsquares}, the divisor $D$ has rational singularities. In particular, $V(x_1^2+x_2^2+x_3^2+x_4^2+x_5^2 +x_6^2,x_1x_2+x_3x_4+x_5x_6)$ has rational singularities. This can be seen another way, using that it is the affine cone over a Fano complete intersection.

\item Let $X\subseteq \A^6$ be the codimension two complete intersection discussed in (2), and let $f=x_1-x_2$. A computer calculation (see Example \ref{eg:FanoCode}) shows that $\widetilde{b}_{(X,f)}(s)=(s+1)^2$, so that $D$ does not have rational singularities.
\end{enumerate}
\end{eg}

\subsection{Thom--Sebastiani property} In this subsection, we prove Theorem \ref{thm-TS}.

For $i=1,2$, let $X_i \subseteq Y_i$ be a closed subvariety of a smooth variety $Y_i$. Let $f_i \in \cO_{Y_i}(Y_i)$ restrict to a nonzerodivisor on $X_i$ and consider the function $f_1+f_2$ on $Y_1 \times Y_2$. Let $c_i = \dim Y_i -\dim X_i$ denote the codimension of the embedding $X_i\subseteq Y_i$. We fix $x_i \in X_i \cap V(f_i) \subseteq Y_i$. We replace $Y_i$ with appropriate neighborhoods of $x_i$ such that $V(f_1+f_2)_{\rm sing} = V(f_1)_{\rm sing} \times V(f_2)_{\rm sing}$, in order to apply \cite{ThomSebastiani}.

We will use the filtered Thom--Sebastiani formula of \cite{ThomSebastiani}*{Thm. 1.2}. It is important to recall that in that paper, the authors used the indexing conventions for \emph{right} filtered $\cD_Y$-modules. This explains the slightly unusual shifts in our version of the formula, but it is a direct consequence of their result. More specifically, we will use \cite{ThomSebastiani}*{(1.2.6)}, which gives an isomorphism for all $k \in \Z$:
\[ {\rm Gr}^F_{(c_1+c_2+1)+k} {\rm Gr}_V^\alpha(\tcB_{f_1+f_2}) \cong \bigoplus_{\alpha_1 \in (-1,0], p\in \Z} {\rm Gr}^F_{(c_1+1)+p} {\rm Gr}_V^{\alpha_1}(\tcB_{f_1}) \boxtimes {\rm Gr}^F_{(c_2+1)+(k-p)} {\rm Gr}_V^{\alpha-\alpha_1}(\tcB_{f_2}).\]

It is easy (by the $(F,V)$-bifiltered isomorphisms $\de_t^j$ on the microlocal modules) to rewrite the right-hand side as
\[ {\rm Gr}^F_{(c_1+c_2+1)+k} {\rm Gr}_V^\alpha(\tcB_{f_1+f_2}) \cong \bigoplus_{\alpha_1 \in (-1,0], p\in \Z} {\rm Gr}^F_{c_1+1} {\rm Gr}_V^{\alpha_1+p}(\tcB_{f_1}) \boxtimes {\rm Gr}^F_{c_2+1} {\rm Gr}_V^{\alpha-\alpha_1+(k-p)}(\tcB_{f_2}).\]

Now, as $F_{c_1+1} \tcB_{f_1} \subseteq V^{>0}\tcB_{f_1}$, we only need to take the sum over $\alpha_1 + p > 0$. If we write $\lambda = \alpha_1 + p$, we can rewrite the right-hand side as
\begin{equation} \label{eq-TSFiltered} {\rm Gr}^F_{(c_1+c_2+1)+k} {\rm Gr}_V^\alpha(\tcB_{f_1+f_2}) \cong \bigoplus_{\lambda > 0} {\rm Gr}^F_{c_1+1} {\rm Gr}_V^{\lambda}(\tcB_{f_1}) \boxtimes {\rm Gr}^F_{c_2+1} {\rm Gr}_V^{\alpha+k - \lambda}(\tcB_{f_2}).\end{equation}

\begin{proof}[Proof of Theorem \ref{thm-TS}] Recall that the equivalent way to describe $\widetilde{\alpha}(X_i,f_i)$ is as
\[ \inf\{ \lambda \in \Q \mid {\rm Gr}^F_{c_i+1} {\rm Gr}_V^\lambda(\tcB_{f_i}) \neq 0\},\]
and the local version can be computed via the stalk
\begin{equation}\label{e:localexp}
\widetilde{\alpha}_{x_i}(X_i,f_i) = \inf\{ \lambda \in \Q \mid {\rm Gr}^F_{c_i+1} {\rm Gr}_V^\lambda(\tcB_{f_i})_{x_i} \neq 0\},
\end{equation}
where we view ${\rm Gr}^F_{c_i+1} {\rm Gr}_V^\lambda(\tcB_{f_i})$ as a coherent sheaf on $X_i$. By analytification, the (non)-vanishing of this stalk is equivalent to the corresponding (non)-vanishing of the analytic stalk. The isomorphism \eqref{eq-TSFiltered} also holds at the analytic stalk level: indeed, \cite{ThomSebastiani} is written in the analytic setting in the first place, and the stalk at $(x_1,x_2)$ can be computed by the direct limit over neighborhoods of the form  $U_1\times U_2$ where $x_i \in U_i$ is an analytic neighborhood of $x_i$ in $Y_i$. This is why we need to analytify in the first place.

Now, taking $k = 0$ in (the analytic stalk version of) \eqref{eq-TSFiltered}, and using the description (\ref{e:localexp}) of the minimal exponent, we conclude that ${\rm Gr}^F_{c_1+c_2+1} {\rm Gr}_V^\alpha(\tcB_{f_1+f_2})_{(x_1,x_2)} \neq 0$ implies $\lambda \geq \widetilde{\alpha}_{x_1}(X_1,f_1)$ and $\alpha - \lambda \geq \widetilde{\alpha}_{x_2}(X_2,f_2)$. By combining these inequalities, we conclude $\alpha \geq \widetilde{\alpha}_{x_1}(X_1,f_1) + \widetilde{\alpha}_{x_2}(X_2,f_2)$.

On the other hand, if we take $\lambda = \widetilde{\alpha}_{x_1}(X_1,f_1)$ and $\alpha = \widetilde{\alpha}_{x_1}(X_1,f_1) + \widetilde{\alpha}_{x_2}(X_2,f_2)$, we easily see that the right-hand side is nonzero, which proves the claim.
\end{proof}

\section{General Linear Combination}\label{sec:linearCombo}

In this section, we prove Theorem \ref{thmx:linearCombo} relating Bernstein--Sato polynomials of ideals $\fra\subseteq \mathcal{O}_X$ to microlocal $b$-functions of general linear combinations.

Let $X$ be a subvariety of a smooth variety $Y$ with $f_1,\dots, f_r \in \cO_Y(Y)$ restricting to $f_1,\dots, f_r$ on $X$, generating the ideal $\fra$. By abuse of notation we also let $\fra \subseteq \cO_Y$ denote the ideal defined by $f_1,\dots, f_r$. Let $Y' = Y \times \A^r_y$ be the affine space over $Y$, which contains $X\times \A^r_y$ as a closed subvariety.

\begin{lem} We have
\[ {\rm IC}_{X\times \A^r_y}^H = {\rm IC}_X^H \boxtimes \Q_{\A^r}^H[r].\]
\end{lem}
\begin{proof} 
This is a consequence of a more general fact that 
\[ {\rm IC}_{X\times Z}^H = {\rm IC}_X^H \boxtimes {\rm IC}_Z^H,\]
which follows from the equality
\[ \Q_{X\times Z}^H[\dX+\dZ] = \Q_X^H[\dX]\boxtimes \Q_Z^H[\dZ],\]
and taking $\cH^0(-)$ of both sides, which, since all objects lie in $D^{\leq 0}({\rm MHM}(-))$, gives equality of mixed Hodge modules
\[ \cH^0(\Q_{X\times Z}^H[\dX+\dZ]) = \cH^0(\Q_X^H[\dX]) \boxtimes \cH^0(\Q_Z^H[\dZ]).\]
    The claim follows by taking ${\rm Gr}^W_{\dX+\dZ}(-)$ on both sides and using the usual formula for weight filtration on an exterior product.
\end{proof}

Thus, if $\cM$ is the $\cD_Y$-module underlying ${\rm IC}_X^H$, we see that $\cM[y_1,\dots, y_r]$ is the $\cD_{Y \times \A^r_y}$-module underlying ${\rm IC}_{X\times \A^r_y}^H$. Moreover, the lowest Hodge piece of ${\rm IC}_{X\times \A^r_y}^H$ is related to the lowest Hodge piece of ${\rm IC}_{X}^H$ by the formula:
\begin{equation}\label{e:hodgeOnBox}
F_{c }(\cM[y_1,\dots, y_r]) = (F_c \cM)[y_1,\dots, y_r].    
\end{equation} 

Let $\Gamma \colon Y \to Y \times \A^r_t$ be the graph embedding along $f_1,\dots, f_r$, and similarly let $\gamma \colon Y\times \A^r_y \to Y\times \A^r_y \times \A^1_z$ be the graph embedding along $h = \sum_{i=1}^r y_i f_i$. 

As in the definition of the $b$-function, we consider the modules $\cB_f = \Gamma_+ \cM = \bigoplus_{\alpha \in \N^r} \cM \de_t^\alpha \delta_f$ and $\cB_h = \gamma_+ \cM[y_1,\dots, y_r] = \bigoplus_{k \geq 0} \cM[y_1,\dots, y_r] \de_z^k \delta_h$. In fact, we will consider the microlocal version of the latter module:
\[ \widetilde{\cB}_h = (\gamma_+ \cM[y_1,\dots, y_r])[\de_z^{-1}] = \bigoplus_{k\in \Z}  \cM[y_1,\dots, y_r] \de_z^k \delta_h.\]

We define $\widetilde{\cB}_h^{(k)} = \bigoplus_{\alpha \in \N^r} \cM y^\alpha \de_z^{|\alpha|-k} \delta_h$, so that we have a direct sum decomposition
\[ \widetilde{\cB}_h = \bigoplus_{k\in \Z} \widetilde{\cB}_h^{(k)}. \]

It is easy to check that 
\[y_i \widetilde{\cB}^{(k)}_h \subseteq \widetilde{\cB}^{(k+1)}_h,\quad \de_{y_i} \widetilde{\cB}^{(k)}_h \subseteq \widetilde{\cB}^{(k-1)}_h,\quad \de_z^\ell \widetilde{\cB}^{(k)}_h = \widetilde{\cB}^{(k-\ell)}_h.\]

Importantly,
\begin{equation} \label{eq-EulerOp} s = \theta_y = \sum_{i=1}^r y_i\de_{y_i} \text{ acts on } \widetilde{\cB}_h^{(0)}.\end{equation}

As above, the microlocal $V$-filtration on $\cD_{Y\times \A^r_y \times \A^1_z}[\de_z^{-1}]$ is defined so that $\cD_{Y\times \A^r_y}$ has degree $0$, $z,\de_z^{-1}$ have degree $1$ and $\de_z$ has degree $-1$. Thus, $s = -\de_z z$ has degree $0$.

By \cite{RadonFourier}*{Prop. 4.3} we have a $V$-filtered, $\cD_Y$-linear isomorphism
\[ \varphi \colon ( \cB_f,V)\cong (\widetilde{\cB}_h^{(0)},V), \]
\[ \varphi( m \de_t^\alpha \delta_f) = m y^\alpha \de_z^{|\alpha|} \delta_h,\]
such that
\[ \varphi \circ \de_{t_i} = (y_i\de_z) \circ \varphi,\quad \varphi \circ t_i = (-\de_{y_i} \de_z^{-1}) \circ \varphi.\]

We see that $\varphi \circ s = \theta_y \circ \varphi = s \circ \varphi$, using $s = \sum_{i=1}^r -\de_{t_i }t_i$ on $\cB_f$ and (\ref{eq-EulerOp}).

\noindent We are now ready to prove the main result of the section.

\begin{proof}[Proof of Theorem \ref{thmx:linearCombo}] Let $m_1,\dots, m_{\ell}$ be $\cO_X$-generators of $F_{c} \cM$. Thus, $m_1,\dots, m_{\ell}$ are also $\cO_{X\times \A^r_y}$-generators of $F_{c}(\cM[y_1,\dots, y_r]) = (F_{c} \cM)[y_1,\dots, y_r]$.

By definition, $b_{(X,\fra)}(s)$ is the monic polynomial of least degree such that, for $1\leq i\leq \ell$, we have
\[
b_{(X,\fra)}(s) m_i \delta_f \in V^1 \cD_{Y \times \A^r_t} \cdot F_{c} \cM.    
\] 

Similarly, $\widetilde{b}_{(X\times \A^r_y,h)}(s)$ is the monic polynomial of least degree such that, for $1\leq i \leq \ell$, we have
\[ \widetilde{b}_{(X\times \A^r_y,h)}(s) m_i \delta_h \in V^1 (\cD_{Y \times \A^r_y \times \A^1_z}[\de_z^{-1}]) \cdot (F_{c}(\cM[y_1,\dots, y_r])\delta_h)\]
\[= V^1 (\cD_{Y \times \A^r_y \times \A^1_z}[\de_z^{-1}]) \cdot (F_{c} \cM\delta_h).\]

In such a functional equation, since the left hand side lies in $\widetilde{\cB}^{(0)}_h$, we know that we can take the operator $P \in V^1 (\cD_{Y\times \A^r_y \times \A^1_z}[\de_z^{-1}])$ to lie in the weight $0$ part. We have
\[ V^1(\cD_{Y\times \A^r_y \times \A^1_z}[\de_z^{-1}]) = \cD_{Y\times \A^r_y}[s,\de_z^{-1}] \de_z^{-1},\]
and so the weight $0$ part corresponds to the weight $-1$ part on $\cD_{Y\times \A^r_y}[s,\de_z^{-1}]$, which we can write as
\[ \sum_{j=1}^r (\cD_{Y\times \A^r_y}[s,\de_z^{-1}])^0 \de_{y_j}.\]

In summary, $\widetilde{b}_{(X\times \A^r_y,h)}(s)$ is the monic polynomial of least degree such that, for $1\leq i \leq \ell$, we have
\[ \widetilde{b}_{(X\times \A^r_y,h)}(s) m_i \delta_h \in \sum_{j=1}^r (\cD_{Y\times \A^r_y}[s,\de_z^{-1}])^0 (\de_{y_j} \de_z^{-1}) \cdot (F_{c}\cM\delta_h).\]
Now it is immediate from the properties of the morphism $\varphi$ that we have equality of the two $b$-functions, as claimed.
\end{proof}

\section{Algorithms for complete intersections with rational singularities}\label{sec:Algs}

In this section, we work in the affine setting. Let $Y=\A^n$ with coordinate ring $S=\C[x_1,\ldots,x_n]$. Let $X\subseteq \A^n$ be a complete intersection of codimension $c$ with defining equations $g_1,\ldots, g_c\in S$. We assume that $X$ has rational singularities.

\subsection{Calculating b-functions for hypersurfaces}

Let $f\in S$ be a polynomial that restricts to a nonzero noninvertible function on $X$. In this section, we describe an algorithm for $b_{(X,f)}(s)$.

For ease of notation, we write $G=\prod_{i=1}^c g_i$. Since $X$ is a complete intersection, we have the following explicit description of $\cH^c_X(S)$:
\begin{equation}\label{eqn:lcX}
\cH^c_X(S) = \frac{S_G}{\sum_{i=1}^c S_{g_1\cdots \widehat{g_i}\cdots g_c}}.   
\end{equation}
We have ${\rm IC}_X^H(-c)=W_{n+c} \cH^c_X(S)$ and, since $X$ is a complete intersection with rational singularities, we have \cites{olano23, cdm24}:
\begin{equation}\label{eqn:F0ratsing}
F_0({\rm IC}_X^H(-c))= S\cdot \left[\frac{1}{G}
\right]\subseteq \cH^c_X(S),
\end{equation}
where $[1/G]$ denotes the class of $1/G\in S_G$ via the identification (\ref{eqn:lcX}).

Let $\cD$ be the Weyl algebra on $\A^n$, and let $\Gamma \colon \A^n \to \A^n \times \A^1_t$ be the graph embedding along $f$. We let $M$ denote the $\cD$-module underlying ${\rm IC}_X^H(-c)$ and consider $ \Gamma_+ M = \bigoplus_{k\geq 0} M \de_t^k\delta_f$. We let $s = -\de_t t$. Using (\ref{eqn:F0ratsing}) and the definition of $b_{(X,f)}(s)$, we obtain the following result.

\begin{prop}\label{prop:eqnCIratsing}
Let $X\subseteq \A^n$ be a complete intersection with rational singularities. The polynomial $b_{(X,f)}(s)$ is the monic polynomial $b(s)$ of minimal degree satisfying 
\begin{equation}
P(s)\cdot\left[\frac{f}{G} \right]\delta_f = b(s)\left[\frac{1}{G} \right] \delta_f , 
\end{equation}
for some $P(s)\in \cD[s]$.    
\end{prop}

Let $L_G\subseteq \cD$ be the annihilator of $[1/G]\in M$, so that $M=\cD/L_G$. Let $b^{L_G}_f(s)$ denote the Bernstein--Sato polynomial of $L_G$ and $f$ as described in \cite{walther2002}*{Section 2.3}. It follows from Proposition \ref{prop:eqnCIratsing} that, in our setting:
\begin{equation}\label{eqn:equalsWalther}
b_{(X,f)}(s)= b^{L_G}_f(s).   
\end{equation}
Since $f$ restricts to a nonzero noninvertible function on $X$, it follows that $M$ is $f$-torsion free. Therefore, there is an algorithm for $b^{L_G}_f(s)$, see \cite{walther2002}*{Algorithm 3.9}. This functionality is implemented in Macaulay2 \cite{M2}, in the \texttt{BernsteinSato} package, via the function \texttt{globalB}. In particular, via (\ref{eqn:equalsWalther}), one may calculate $b_{(X,f)}(s)$ using this function, provided one has a description of $L_G$. To obtain this, we compare $L_G$ to the annihilator of $1/G\in S_G$.

\begin{lem}\label{lem:ann1G}
Using notation as above, we have
\begin{equation}\label{eqn:ann1G}
\operatorname{Ann}_{\cD}\left(\left[\frac{1}{G}
\right]\right)=\operatorname{Ann}_{\cD}\left(\frac{1}{G} \right)+\cD\cdot \left\langle g_1,\ldots , g_c\right\rangle.
\end{equation}
\end{lem}

\begin{proof}
By (\ref{eqn:lcX}), we see that the right side of (\ref{eqn:ann1G}) is contained in the left side. To prove the reverse inclusion, let $P\in \operatorname{Ann}_{\cD}([1/G])$. Then there exists $h_1,\ldots, h_c\in S$ such that
$$
P\cdot \frac{1}{G}=\sum_{i=1}^c \frac{h_i}{g_1\cdots \widehat{g_i}\cdots g_c}.
$$
In particular, $P-\sum_{i=1}^c h_ig_i\in \operatorname{Ann}_{\cD}(1/G)$. Therefore, $P\in \operatorname{Ann}_{\cD}(1/G)+\cD\cdot \langle g_1,\ldots,g_c\rangle$.
\end{proof}

By Lemma \ref{lem:ann1G}, we obtain a presentation of $L_G$ from a presentation of $\operatorname{Ann}_{\cD}\left(1/G \right)$, the latter of which is implemented in the \texttt{BernsteinSato} package of Macaulay2, via the function \texttt{rationalFunctionAnnihilator}. 

Altogether, we obtain the following algorithm for $b_{(X,f)}(s)$.

\begin{algorithm}\label{alg:ratCI}
\;

\noindent INPUT: Polynomials $g_1,\ldots,g_c,f\in \C[x_1,\ldots, x_n]$ such that $g_1,\ldots,g_c$ form a regular sequence defining $X\subseteq \A^n$, $f$ restricts to a nonzero noninvertible function on $X$, and $X$ has rational singularities.

\noindent OUTPUT: The Bernstein--Sato polynomial $b_{(X,f)}(s)$.

\smallskip

Let $\cD$ be the Weyl algebra on $\A^n$ and let $G=g_1g_2\cdots g_c$.

\begin{enumerate}
 \item Determine $\operatorname{Ann}_{\cD}(1/G)$ via \cite{walther2002}*{Remark 3.16}. 

 \smallskip

 \item Set $L_G=\operatorname{Ann}_{\cD}(1/G)+\cD\cdot \langle g_1,\ldots, g_c\rangle$.

 \smallskip

 \item Calculate $b^{L_G}_f(s)$ following \cite{walther2002}*{Algorithm 3.9}. Set $b_{(X,f)}(s)=b^{L_G}_f(s)$.
\end{enumerate}

\smallskip

\noindent END.
\end{algorithm}

We carry out an example using Macaulay2.

\begin{eg}
Let $S=\C[x_{i,j}]_{1\leq i,j\leq 3}$, let $g$ be the $3\times 3$ determinant $g=\det(x_{i,j})$, and let $f=x_{1,1}$. We load \texttt{BernsteinSato} and create $g$ and $f$ as elements of the Weyl algebra:

\smallskip

\begin{verbatim}
i1: needsPackage "BernsteinSato";
i2: S = QQ[x_(1,1)..x_(3,3)];
i3: D = makeWeylAlgebra S;
i4: g = determinant genericMatrix(D,x_(1,1),3,3);
i5: f = x_(1,1);
\end{verbatim}

\smallskip

\noindent Next, we calculate the $\cD$-annihilators of $1/g$ and $[1/g]$:

\smallskip

\begin{verbatim}
i6: I = rationalFunctionAnnihilator(g);
i7: LG = I + ideal(g);
\end{verbatim}

\smallskip

\noindent Finally, we use \texttt{globalB} to obtain the Bernstein--Sato polynomial in factored form:

\smallskip

\begin{verbatim}
i8: b = globalB(LG,f);   
i9: factorBFunction(b#Bpolynomial)
o9: (s + 1)(s + 2)
\end{verbatim}

\smallskip

\noindent We find $b_{(X,f)}(s)=(s+1)(s+2)$ for $X=V(g)$. In particular, $\cD\cdot [1/g]f^{-2}=({\rm IC}_X)_f$. After calculating $\mathtt{b=globalB(LG,f)}$, one may obtain the operator $P(s)$ using $\mathtt{b\#Boperator}$, though it has too many terms to print here.
\end{eg}

Next, we carry out an example when $X$ has higher codimension.

\begin{eg}\label{eg:FanoCode}
Let $S=\C[x_1,x_2,x_3,x_4,x_5,x_6]$ and let $X=V(g_1,g_2)\subseteq \A^6$, where
$$
g_1=x_1^2+x_2^2+x_3^2+x_4^2+x_5^2+x_6^2,\quad g_2=x_1x_2+x_3x_4+x_5x_6.
$$
By Example \ref{eg:ratSing}, $X$ has rational singularities. We calculate $b_{(X,f)}(s)$ when $f=x_1+x_2$:

\smallskip

\begin{verbatim}
i2: S = QQ[x_1..x_6]; 
i3: D = makeWeylAlgebra S; 
i4: Ggens = {x_1^2+x_2^2+x_3^2+x_4^2+x_5^2+x_6^2, x_1*x_2+x_3*x_4+x_5*x_6};
i5: f = x_1+x_2; 
i6: LG = rationalFunctionAnnihilator(product(Ggens)) + ideal(Ggens); 
i7: b = globalB(LG,f); 
i8: factorBFunction(b#Bpolynomial) 
          2
o8: (s + 1)  
\end{verbatim}

\smallskip

\noindent So $b_{(X,f)}(s)=(s+1)^2$ in this example. In particular, $\cD\cdot [1/G]f^{-1}=({\rm IC}_X)_f$.
\end{eg}

\begin{rmk}
One can check if a complete intersection has rational singularities using the function \texttt{hasRationalSing}, found in the \texttt{BernsteinSato} package.    
\end{rmk}

\begin{rmk}
One could try to generalize Algorithm \ref{alg:ratCI} to the situation where $X$ does not necessarily have rational singularities. The next case would be when $F_0({\rm IC}_X)\subseteq \cH^c_X(S)$ is a cyclic $S$-module, generated by $[h/G]$ for some $h\in S$. In this case, $b_{(X,f)}(s)=b^L_f(s)$, where $L$ is the $\cD$-annihilator of $[h/G]\in \cH^c_X(S)$. One could calculate the annihilator of $h/G\in S_G$ using \texttt{rationalFunctionAnnihilator}, but it is unclear how to precisely relate this to $L$.
\end{rmk}

\begin{rmk}
Assume $F_0 {\rm IC}_X^H$ is $S$-generated by $m_1,\dots, m_a$, and let $b_{m_i}(s)$ be the monic polynomial of least degree such that
\[ b_{m_i}(s) \cdot m_i \delta_f \in V^1 \cD_{Y\times \A^1_t} \cdot (m_i \delta_f).\] 
It is straightforward to show that $b_{(X,f)}(s)$ divides ${\rm lcm}_{1\leq i\leq a} b_{m_i}(s)$.
\end{rmk}

\subsection{Towards calculation of the microlocal b-function for hypersurfaces}

This subsection consists of preliminaries towards the calculation of the microlocal $b$-function. For now, we let $X$ be an irreducible subvariety of a smooth variety $Y$, and we assume $f\in \cO_Y$ restricts to a nonzero, noninvertible function on $X$.

We fix $m$ to be any element of ${\rm IC}_X$. We define $b_{(m,f)}(s)$ (resp. $\widetilde{b}_{(m,f)}(s)$) to be the monic polynomial of minimal degree satisfying 
\begin{equation}\label{eqn:bFunctionsElement}
b(s)m\delta_f \in \cD_Y[s,t]t\cdot m\delta_f,\quad \textnormal{resp.}\quad b(s) m\delta_f \in \cD_Y[s,t,\partial_t^{-1}]\de_t^{-1}\cdot m\delta_f.   
\end{equation}

In the next subsection, we will assume that $X$ is a complete intersection with rational singularities, and that $m$ is the cyclic $\cO_Y$-generator of $F_c {\rm IC}_X$. We provide a method to obtain $\widetilde b_{(X,f)}(s)=\widetilde{b}_{(m,f)}(s)$ from $b_{(X,f)}(s)=b_{(m,f)}(s)$ via a sequence of intermediate polynomials obtained by iteratively dividing linear factors. This allows us to circumvent the fact that the localized ring $\widetilde{\cR}$ is not implemented in algebra software such as Macaulay2 or Singular.

For $k,\ell$ nonnegative integers, we write
\begin{equation}
\binom{s+\ell}{k}=\frac{(s+\ell)\cdots (s+\ell-k+1)}{k!}\in \C[s],    
\end{equation}
so that $\binom{s+\ell}{0}=1$. With this notation, we have the following.

\begin{lem} \label{lem-functionalEquationMicrolocal}
Using notation as above, $\widetilde{b}_{(m,f)}(s)$ is the monic polynomial of minimal degree such that there exists a $d \geq 0$ and a functional equation:
\small
\begin{equation}\label{eqn:microFunctional}
b(s) (s+1)\dots (s+d+1)m \delta_f = \sum_{j=0}^d P_j(s) \binom{s+d+1}{d-j} m t^{j+1} \delta_f.    
\end{equation}

\end{lem}

In particular, the $d$-th summand of (\ref{eqn:microFunctional}) is $P_d(s)mt^{d+1}\delta_f$.

\begin{proof}
Up to sign, we have
\[ t^k \de_t^k = (s+1)\dots (s+k),\]
and so if we start with a microlocal functional equation as in (\ref{eqn:bFunctionsElement}):
\[ b(s) m \delta_f = P_0(s)m\de_t^{-1} \delta_f + P_1(s) m \de_t^{-2} \delta_f + \dots + P_d(s) m\de_t^{-d-1}\delta_f,\]
and we multiply through by $t^{d+1} \de_t^{d+1}$, we get 
\[ b(s) (s+1)\dots (s+d+1)m \delta_f = \sum_{j=0}^{d} P_j(s) t^{d+1} \de_t^{d+1-(j+1)} m \delta_f.\]
We obtain (\ref{eqn:microFunctional}) using that $t^{d+1} \de_t^{d-j} = t^{j+1} t^{d-j} \de_t^{d-j} = t^{j+1} (s+1) \dots (s+d-j)$.    
\end{proof}

Next, we explain some polynomials which approximate the microlocal $b$-function. For $d\geq 0$, let $c_d(s)$ be the minimal polynomial of the functional equation
\begin{equation}\label{eqn:cFunction}
c(s)m \delta_f = \sum_{j=0}^d P_j(s) \binom{s+d+1}{d-j} m t^{j+1} \delta_f.    
\end{equation} 
We have the following divisibility relations.

\begin{lem}\label{lem:cLemma}
For $d\geq 0$, the following is true about $c_d(s)$:  
\begin{enumerate}
\item $(s+1)(s+2) \cdots (s+d+1)$  divides $c_d(s)$,

\item $c_d(s)$ divides $c_{d+1}(s)$.
\end{enumerate}
\end{lem}

\begin{proof}
(1) Let $1\leq a\leq d+1$. Then we have the following expression in $({\rm IC}_X)_f$:
  \[ c_d(-a)m f^{-a} = \sum_{j=0}^d P_j(-a) \binom{-a+d+1}{d-j} m f^{-a+j+1}.\]
  If $a\geq j+2$, then the $j$-th summand is zero. Hence, this expression is equal to:
  \[ \sum_{j=a-1}^d P_j(-a) \binom{-a+d+1}{d-j} m f^{-a+j+1}.\]
If $j\geq a-1$, then $-a+j+1\geq 0$. Therefore, this element of $({\rm IC}_X)_f$ belongs to ${\rm IC}_X\subsetneq ({\rm IC}_X)_f$. On the other hand, $c_d(-a)m f^{-a}$ belongs to $({\rm IC}_X)_f\setminus {\rm IC}_X$ unless it is zero. It follows that $c_d(-a)=0$. 

(2) For $d\geq 0$ consider the following ideal in $\cD[s]$:
\[
I_d=\cD[s]\langle (s+2)\dots (s+d+1) f,\;  (s+3)\dots (s+d+1)f^2, \dots ,(s+d+1)f^d,\; f^{d+1}\rangle.
\]
Then $c_d(s)$ is the monic generator of $\C[s]\cap (\operatorname{Ann}_{\cD[s]}(m\delta_f)+I_d)$. Since $I_d\supseteq I_{d+1}$, we have $c_d(s)$ divides $c_{d+1}(s)$.
\end{proof}

For $d\geq 0$ we define the polynomial
\begin{equation}\label{eqn:bdFunction}
b_d(s)=\frac{c_d(s)}{(s+1)(s+2) \cdots (s+d+1)}.    
\end{equation}
This is the monic polynomial of least degree satisfying (\ref{eqn:microFunctional}). It follows that $\widetilde{b}_{(m,f)}(s) \mid b_d(s)$.

\begin{eg} For $d=0$, the functional equation (\ref{eqn:microFunctional}) is
\[ b(s)(s+1) mf^s = P(s) m f^{s+1},\]
so $b_0(s) = b(s)/(s+1)$. Hence, $b_0(s)$ is always $b_{(m,f)}(s)/(s+1)$, the ``reduced'' $b$-function.

For $d=1$, the functional equation is
\[ b(s) (s+1)(s+2)m f^s = P_0(s) (s+2) m f^{s+1} + P_1(s) m f^{s+2}.\]

For $d=2$, the functional equation is
\[ b(s)(s+1)(s+2)(s+3) m f^s = P_0(s)(s+2)(s+3) m f^{s+1} + P_1(s)(s+3) m f^{s+2} + P_2(s) m f^{s+3}.\]
\end{eg}

The following result explains how to obtain $\widetilde{b}_{(m,f)}(s)$ from the polynomials $b_d(s)$.

\begin{prop}\label{prop:microBound}
Using notation as above:
\begin{enumerate}
\item $b_{d+1}(s)$ divides $b_d(s)$ for  all $d\geq 0$, 

\item $b_d(s)$ divides $b_{d+1}(s)(s+d+2)$ for  all $d\geq 0$,

\item Let $d'$ be the negative of the smallest integer root of $b_{(m,f)}(s)$. Then
\[
\widetilde{b}_{(m,f)}(s) = \gcd_{d\geq 0} (b_d(s)) = b_{d'-1}(s).
\]
\end{enumerate}
\end{prop}

\begin{proof}
(1) Since $b_d(s)$ satisfies the functional equation (\ref{eqn:microFunctional}), multiplying through by $(s+d+2)$, we get
\[ b_d(s) (s+1) \dots (s+d+1)(s+d+2) m\delta_f = \sum_{j=0}^d P_j(s)(d+1-j) \binom{s+d+2}{d+1-j} m t^{j+1}\delta_f,\]
which is a functional equation for $b_{d+1}(s)$ but with $P_{d+1}(s) = 0$, proving the divisibility.

(2) By Lemma \ref{lem:cLemma}, we have $c_d(s)$ divides $c_{d+1}(s)$. Using (\ref{eqn:bdFunction}), it follows that $b_d(s)$ divides $b_{d+1}(s)(s+d+2)$.

(3) We need to show that $b_e(s)=b_{d'-1}(s)$ for all $e> d'-1$. If $e> d'-1$, then $-e\leq -d'$, so $-e-1$ is not a root of $b_{(m,f)}(s)$, hence not a root of $b_d(s)$ for all $d\geq 0$. By (1) and (2) we have
\[
b_e(s)\mid b_{e-1}(s)\mid b_e(s)(s+e+1).
\]
Since $-e-1$ is not a root of $b_{e-1}(s)$, we conclude that $b_e(s)=b_{e-1}(s)$.
\end{proof}

\begin{rmk}
The bound $d'$ in Proposition \ref{prop:microBound} is not always sharp. For instance, when $X$ is smooth and $m=1\in {\rm IC}_X^H=\cO_X$, then $b_{(m,f)}(s)=b_f(s)$ and $\widetilde{b}_{(m,f)}(s)=\widetilde{b}_f(s)=b_0(s)$ \cite{SaitoMicrolocal}. In particular, if $b_f(s)$ has an integral root less than $-1$, then the bound $d'$ is not optimal.

On the other hand, when $X=V(x_1^2+x_2^2+x_3^2+x_4^2+x_5^2+x_6^2)\subseteq \A^6$, $m$ is the cyclic $\cO_Y$-generator of $F_c({\rm IC}_X^H)$, and $f=x_1$, we have 
\[
b_{(m,f)}(s)=(s+1)(s+4),\quad b_0(s) = b_1(s) = b_2(s) = (s+4),\quad \widetilde{b}_{(m,f)}(s)=b_3(s)=1,
\]
so the bound of $d'=4$ is optimal in this case.

\end{rmk}

\subsection{Calculating the microlocal b-function for hypersurfaces}

We return to the situation where $X\subseteq \A^n$ is a complete intersection with rational singularities. Suppose that $f\in S$ restricts to a nonzero noninvertible function on $X$.

Using (\ref{eqn:bdFunction}) and Proposition \ref{prop:microBound} we obtain the following algorithm for the microlocal $b$-function of $(X,f)$.

\begin{algorithm}\label{alg:microRatCI}
\;

\noindent INPUT: Polynomials $g_1,\ldots,g_c,f\in \C[x_1,\ldots, x_n]$ such that $g_1,\ldots,g_c$ form a regular sequence defining $X\subseteq \A^n$, $f$ restricts to a nonzero noninvertible function on $X$, and $X$ has rational singularities.

\noindent OUTPUT: The microlocal Bernstein--Sato polynomial $\widetilde{b}_{(X,f)}(s)$.

\smallskip

Let $\cD$ be the Weyl algebra on $\A^n$ and let $G=g_1g_2\cdots g_c$.

\begin{enumerate}
 \item Determine $L_G=\operatorname{Ann}_{\cD}(1/G)+\cD\cdot \langle g_1,\ldots, g_c\rangle$ as in Algorithm \ref{alg:ratCI}. 

 \smallskip

 \item Calculate $J^{L_G}(f^s)=\operatorname{Ann}_{\cD[s]}([1/G]f^s)$ using $L_G$ via \cite{walther2002}*{Algorithm 3.7}.

 \smallskip

 \item Calculate $b_{(X,f)}(s)=b^{L_G}_f(s)=(J^{L_G}(f^s)+\cD[s]f)\cap \C[s]$ using elimination. 

 \smallskip

 \item Set $d'$ to be the negative of the smallest integer root of $b_{(X,f)}(s)$.

 \smallskip

 \item Calculate the monic generator $c_{d'-1}(s)$ of $(J^{L_G}(f^s)+I_{d'-1})\cap \C[s]$ via elimination, where
 \[
I_d:=\cD[s]\langle (s+2)\dots (s+d+1) f,\;  (s+3)\dots (s+d+1)f^2, \dots ,(s+d+1)f^d,\; f^{d+1}\rangle.
\]

\smallskip

\item Set $\widetilde{b}_{(X,f)}(s)=b_{d'-1}(s)=c_{d'-1}(s)/((s+1)(s+2)\cdots (s+d'))$.
\end{enumerate}

\smallskip

\noindent END.
\end{algorithm}

One may calculate $J^{L_G}(f^s)=\operatorname{Ann}_{\cD[s]}([1/G]f^s)$ in Macaulay2 using $\mathtt{AnnIFs}$.

\begin{eg}
We calculate the microlocal $b$-function of $X=V(x^3+y^3+z^3+w^3)\subseteq \A^4$ along the hypersurface $f=x-y$.

\smallskip

\begin{verbatim}
i1: needsPackage "BernsteinSato";
i2: S = QQ[x,y,z,w];
i3: D = makeWeylAlgebra S;
i4: g = x^3+y^3+z^3+w^3;
i5: LG = rationalFunctionAnnihilator(g) + ideal(g);
i6: f = x-y;
i7: JLGfs = AnnIFs(LG,f);
i8: b = factorBFunction ( globalB(LG,f) )#Bpolynomial 
           2
o8 = (s + 1) (s + 2)(s + 3)(s + 4)
\end{verbatim}

\smallskip

\noindent Thus, $d'=4$ here. Proceeding with steps (4)--(6) of Algorithm \ref{alg:microRatCI} we have

\smallskip

\begin{verbatim}
i9: Ds = ring JLGfs; ns = numgens Ds; s = Ds_(ns-1); f = sub(f,Ds);
i10: I = ideal( (s+2)*(s+3)*(s+4)*f, (s+3)*(s+4)*f^2, (s+4)*f^3, f^4 );
i11: importFrom(BernsteinSato, {"eliminateWA"});
i12: cElimDs = eliminateWA( JLGfs + I, drop( gens Ds, -1))
i13: c = factorBFunction ( flatten entries gens sub(cElimDs, QQ[s]) )_0
            2              2
o13 = (s + 1) (s + 2)(s + 3) (s + 4)
\end{verbatim}

\smallskip

\noindent Dividing $c_{d'-1}(s)$ by $(s+1)(s+2)(s+3)(s+4)$ we get 
\[
\widetilde{b}_{(X,f)}(s)=(s+1)(s+3).
\]
In particular, $D=V(x-y, 2x^3+z^3+w^3)$ does not have rational singularities.
\end{eg}

\subsection{Calculating the b-function for a subvariety}

Let $\fra \subseteq S$ be an ideal. By Theorem \ref{thmx:linearCombo} we obtain the following algorithm for $b_{(X,\fra)}(s)$ when $X\subseteq \A^n$ is a complete intersection with rational singularities.

\begin{algorithm}\label{alg:idealRatCI}
\;

\noindent INPUT: Polynomials $g_1,\ldots,g_c,f_1,\ldots,f_r\in \C[x_1,\ldots, x_n]$ such that $g_1,\ldots,g_c$ form a regular sequence defining $X\subseteq \A^n$, $\fra=(f_1,\ldots,f_r)$ restricts to a nonzero ideal on $X$, and $X$ has rational singularities.

\noindent OUTPUT: The Bernstein--Sato polynomial $b_{(X,\fra)}(s)$.

\smallskip

Let $T=\C[x_1,\ldots,x_n,y_1,\ldots,y_r]$ and let $h=f_1y_1+\ldots+f_ry_r\in T$.

\begin{enumerate}
 \item Calculate $\widetilde{b}_{(X\times \A^r,h)}(s)$ following Algorithm \ref{alg:microRatCI}.

 \smallskip

 \item Set $b_{(X,\fra)}(s)=\widetilde{b}_{(X\times \A^r,h)}(s)$.
\end{enumerate}

\smallskip

\noindent END.
\end{algorithm}

\begin{eg}
Let $X=V(x_1^2+x_2^5+x_2x_3^2)\subseteq \A^5$ and let $\fra=\langle x_4,x_5\rangle\subseteq \C[x_1,x_2,x_3,x_4,x_5]$. We expect that $b_{(X,\fra)}(s)$ will reflect the geometry of the $\mathsf{D}_6$ surface singularity 
\[
D=V(x_1^2+x_2^5+x_2x_3^2,x_4,x_5). 
\]
Proceeding with Algorithm \ref{alg:idealRatCI}, we have   

\smallskip

\begin{verbatim}
i1: needsPackage "BernsteinSato";
i2: T = QQ[x_1..x_5,y_1,y_2];
i3: D = makeWeylAlgebra T;
i4: g = x_1^2 + x_2^5 + x_2*x_3^2;
i5: LG = rationalFunctionAnnihilator(g) + ideal(g);
i6: f = x_4*y_1 + x_5*y_2;
i7: JLGfs = AnnIFs(LG,f);
i8: b = factorBFunction ( globalB(LG,f) )#Bpolynomial
o8: (s + 1)(s + 2)
i9: Ds = ring JLGfs; ns = numgens Ds; s = Ds_(ns-1); f = sub(f,Ds);
i10: I = ideal( (s+2)*f, f^2 );
i11: importFrom(BernsteinSato, {"eliminateWA"});
i12: cElimDs = eliminateWA( JLGfs + I, drop( gens Ds, -1))
i13: c = factorBFunction ( flatten entries gens sub(cElimDs, QQ[s]) )_0
                    2
o13 = (s + 1)(s + 2)
\end{verbatim}

\smallskip

\noindent We conclude that $b_{(X,\fra)}(s)=(s+2)$. By \cite{DirksMultiplier}*{Corollary 1.6}, we recover that $D$ has Du Bois singularities.
\end{eg}

\end{document}

%% file: Preamble.tex
\newcounter{intro}

\newtheorem{intro-conjecture}[intro]{Conjecture}
\newtheorem{intro-corollary}[intro]{Corollary}
\newtheorem{intro-theorem}[intro]{Theorem}

\newcommand{\theoremref}[1]{\hyperref[#1]{Theorem~\ref*{#1}}}
\newcommand{\sectionref}[1]{\hyperref[#1]{Section~\ref*{#1}}}
\newcommand{\lemmaref}[1]{\hyperref[#1]{Lemma~\ref*{#1}}}
\newcommand{\definitionref}[1]{\hyperref[#1]{Definition~\ref*{#1}}}
\newcommand{\propositionref}[1]{\hyperref[#1]{Proposition~\ref*{#1}}}
\newcommand{\conjectureref}[1]{\hyperref[#1]{Conjecture~\ref*{#1}}}
\newcommand{\corollaryref}[1]{\hyperref[#1]{Corollary~\ref*{#1}}}
\newcommand{\exampleref}[1]{\hyperref[#1]{Example~\ref*{#1}}}
\newcommand{\remarkref}[1]{\hyperref[#1]{Remark~\ref*{#1}}}
\newcommand{\itemref}[1]{\hyperref[#1]{Item~\ref*{#1}}}
\newcommand{\equationref}[1]{\hyperref[#1]{Equation~(\ref*{#1})}}
\newcommand{\formularef}[1]{\hyperref[#1]{Formula~(\ref*{#1})}}
\newcommand{\conditionref}[1]{\hyperref[#1]{Condition~(\ref*{#1})}}

\theoremstyle{plain}
\newtheorem{thm}{Theorem}[section]

\newtheorem{lem}[thm]{Lemma}
\newtheorem{prop}[thm]{Proposition}
\newtheorem{cor}[thm]{Corollary}

\theoremstyle{definition}
\newtheorem{algorithm}[thm]{Algorithm}
\newtheorem{defi}[thm]{Definition}

\newtheorem{eg}[thm]{Example}

\newtheorem{question}[thm]{Question}

\theoremstyle{remark}
\newtheorem{rmk}[thm]{Remark}

\def\Mustata{Mus\-ta\-\c{t}\u{a}\xspace}

\def\N{{\mathbf N}}
\def\Z{{\mathbf Z}}
\def\Q{{\mathbf Q}}

\def\C{{\mathbf C}}

\def\A{{\mathbf A}}
\def\P{{\mathbf P}}

\def\cB{\mathcal{B}}

\def\cD{\mathcal{D}}

\def\cH{\mathcal{H}}

\def\cJ{\mathcal{J}}

\def\cM{\mathcal{M}}
\def\cN{\mathcal{N}}
\def\cO{\mathcal{O}}

\def\cR{\mathcal{R}}

\def\cV{\mathcal{V}}

\def\tcB{\widetilde{\mathcal{B}}}

\def\tcR{\widetilde{\mathcal{R}}}

\def\fra{\mathfrak{a}}

\def\.{\cdot}
\def\^{\widehat}

\def\de{\partial}

\def\({\left(}
\def\){\right)}

\renewcommand{\and}{ \ \ \text{ and } \ \ }